\documentclass{article}
\usepackage{amsmath}
\usepackage{amsfonts}
\usepackage{amsthm}
\usepackage [english] {babel}
\usepackage {ucs}
\usepackage[utf8x]{inputenc}
\usepackage{color}
\usepackage{caption, graphicx}
\usepackage{dsfont}

\newtheorem{theorem}{Theorem}[section]

\newtheorem{lemma}{Lemma}[section]
\newtheorem{proposition}{Proposition}[section]

\newtheorem{remark}{Remark}[section]

\def\dsp{\displaystyle}

\def\N{\mathbb{N}}

\def\LL{\mathcal{L}}

\def\FF{\mathcal{F}}

\def\PP{\mathcal{P}}

\def\dx{\delta x}
\def\dt{\delta t}
\def\ZT{$\mathcal{Z}-$}

\graphicspath{{images}}

\begin{document}

\title{Discrete transparent boundary conditions for the linearized Green-Naghdi system of equations}
\author{M. Kazakova\thanks{Corresponding author, Institut de Math\'ematiques de Toulouse; UMR5219, Universit\'e de Toulouse; CNRS, UPS IMT, F-31062 Toulouse, Cedex 9, France, e-mail: mkazakov@math.univ-toulouse.fr},
	P. Noble \thanks{Institut de Math\'ematiques de Toulouse; UMR5219, Universit\'e de Toulouse; CNRS, INSA, F-31077 Toulouse, France, email:  pascal.noble@math.univ-toulouse.fr}}
\date{\vspace{-5ex}}
\maketitle

\begin{center}
{\bf Abstract}
\end{center}

{\small In this paper, we introduce artificial boundary conditions for the linearized Green-Naghdi system of equations. The derivation of such continuous (respectively discrete) boundary conditions include the inversion of Laplace transform (respectively $\mathcal{Z}$-transform) and these boundary conditions are in turn non local in time. In the case of continuous boundary conditions, the inversion is done explicitly. We consider two spatial discretisations of the initial system either on a staggered grid or on a collocated grids, both of interest from the practical point of view. We use a Crank Nicolson time discretization. The proposed numerical scheme with the staggered grid permits explicit $\mathcal{Z}$-transform inversion whereas the collocated grid discretization do not. A stable numerical procedure is proposed for this latter inversion. We test numerically the accuracy of the described method with standard Gaussian initial data and wave packet initial data which are more convenient to explore the dispersive properties of the initial set of equations. We used our transparent boundary conditions to  solve numerically the problem of injecting propagating (planar) waves in a computational domain.}

\section{Introduction}
The motion of incompressible and irrotational fluids under the effect of  gravity is mathematically described by the free surface Euler equations. The complexity of this system led to the derivation of various asymptotic models for the water wave problem which are valid for some special physical regimes (see \cite{L} for more details). The well-known shallow water asymptotic is widely applied since the situations where the horizontal length scale is much greater than the vertical length scale are of particular interest to describe water waves in coastal areas. This regime leads, at first order, to the so-called shallow water equations (\cite{SV}) and this simple system being hyperbolic describes lots physical phenomena. Though in order to take into account the dispersive effects that are important in coastal oceanography, one has to consider second order model like the Green-Naghdi model (\cite{GN}).

The dimensional Green Naghdi equations read as
\begin{equation}\label{GN}
\begin{cases}
h_t  + div(h \vec{u}) = 0,\\
\dsp (h\vec{u})_t + div(h \vec{u} \otimes \vec{u} + pI) = 0, \quad p = \frac{gh^2}{2} + \frac{1}{3}  h^2 \ddot{h},
\end{cases}
\end{equation}
where $h$ is a fluid depth, $\vec{u}$ is a depth-averaged horizontal velocity, indexes means the derivation with respect to $t$, $x\in\mathbb{R}^2$, and dot is a material derivative $\dot{h} = h_t + \vec{u} \cdot \nabla h$. The consistency result with Euler equations can be found in \cite{L}. The model \eqref{GN} describes bidirectional propagation of dispersive water waves in the shallow water regime. It is physically more relevant for water wave problem than the unidirectional models like the Korteweg-de Vries equation or the Benjamin-Bona-Mahony equation which only describe small amplitude/unidirection water waves.  



The original system \eqref{GN} is derived and set on the whole space. Though, for practical applications, the area of study is restricted to a bounded domain and one has to prescribe suitable boundary conditions. We focus here on {\it artificial} boundary conditions in order to let waves go out of the computational domain without reflection or to prescribe an incoming wave on a part of the domain. From a mathematical point of view, the problem is set, in both cases, as follows: given a initial data compactly supported, one search for suitable boundary conditions so that the solution computed with these boundary conditions coincide on the bounded domain with the restriction of the solution set on the whole space. One possibility to solve this problem is to compute the solution on a sufficiently large domain with, say, periodic boundary conditions. Though it is cumbersome from a numerical point of view and requires the solution to remain compactly supported for all time. In particular it is untrue for large classes of dispersive equations like the Korteweg-de Vries equation or the Schr\" odinger equation. Moreover, the energy of the  exact solution for the problem set on the whole space is conserved whereas the energy of the restricted solution should decrease. For all these reasons is important to find the suitable boundary conditions, which absorb the energy at the boundaries and lead to a well-posed initial boundary value problem.


A review on different techniques for the construction such conditions for the linear and nonlinear Schr\"{o}dinger equations can be found in \cite{AABES}. For linear equations, the construction of the exact transparent boundary conditions is carried out by using Laplace transform in time and impose boundary conditions so as to obtain finite energy solutions. The inversion of those conditions yields boundary conditions that are in general non local in time. For nonlinear equations, pseudodifferential or paradifferential calculus is needed and provide transparent boundary conditions in the high frequency/short time regime \cite{AABES}. A numerical implementation of these boundary conditions is not straightforward: see e.g. \cite{ZWH} for a discretization of transparent boundary conditions for the Airy equation which requires an approximation of fractional derivatives. An alternative and fruitful approach consists in starting directly from a discretization of the equations set on the whole space and mimic the approach in the continuous case: the Laplace transform is replaced by the $\mathcal{Z}$-transform: see e.g. \cite{BEL-V} for an application of this strategy to the Airy equation. Though in the former paper, the inverse $\mathcal{Z}$-transform can not be carried out explicitly and the authors implement directly the explicit formula of the inverse tranform. This procedure is not stable from a numerical point of view. Recently the same idea provided the appropriate continuous and discrete boundary conditions for others dispersive equations for unidirectional wave propagation such as Benjamin-Bona-Mahoney (BBM) equation \cite{BMN} and mixed KDV-BBM equation \cite{BNS} where an alternative, stable method is introduce to compute the inverse transform. 




In this paper, we focus on a linearized version of \eqref{GN} model about the steady state $(h,u):=(H_0, 0)+(\eta, w)$ with $|(\eta,w)|\ll 1$. In the one dimensional case, this linearized system is written as:
\begin{equation}\label{KN-inisyst}
		\begin{cases}
			\eta_t + w_x = 0,\\
			w_t  + \eta_x - \varepsilon w_{txx} = 0,
		\end{cases}
		\quad x \in \mathbb{R}, t > 0
\end{equation}
where  $\varepsilon > 0$ is a dispersion parameter. We are interested in derivation of discrete transparent boundary conditions for \eqref{KN-inisyst}: they should provide suitable absorbing boundary conditions for the full system (\ref{GN}) for small amplitude waves. For that purpose, we focus on two spatial discretisations by working either on a collocated grid ($\eta,w$ are evaluated at the same points) or on a staggered grid. We use a Crank Nicolson scheme for time discretisation. We then follow a similar strategy than for the derivation of continuous transparent boundary conditions: we apply 
the $\mathcal{Z}-$transform and identify exponentially growing at $\pm\infty$. By restricting our attention to finite energy solutions, we impose conditions at the boundary points and then apply either explicitly or numerically the inverse $\mathcal{Z}-$transform.  These conditions are generically non local in time and can be cumbersome from a numerical point of view. There are various strategies to implement efficiently those (DTBC). Let us mention in particular ``sum of exponentials'' techniques: this approach is well documented. See e.g. \cite{Arn}, \cite{AES}, for quantum evolution equations and \cite{BEL-V} for an application in the case of the linearized (KdV) equation. 

The paper is organised as follow. In section \ref{EX}, we apply the technique found in \cite{ZWH}  to construct the exact boundary conditions for the linear system \eqref{KN-inisyst}. Moreover one can notice that the system \eqref{KN-inisyst} is equivalent to a linearized version of the Boussinesq equation:
\begin{equation*}
(w - \varepsilon w_{xx} )_{tt} - w_{xx} = 0, 	\quad \forall x \in \mathbb{R}, \forall t > 0,
\end{equation*}
and we focus on the construction boundary conditions for this equation too. It is useful when we construct the discrete conditions for Crank Nicolson time-discretization on a staggered grid: see section \ref{Disc1eq}. As it was already mentioned procedure of discrete boundary conditions construction involve the inversion of non-local in time operator $\mathcal{Z}$-transform, and the main reason to consider the scheme on the staggered grid is that this inversion can be done explicitly. The inversion of conditions for scheme on a collocated grid needs to be done numerically, and a more sophisticated procedure of inversion is presented in section \ref{Disc}. Finally, in section \ref{Num}, we present some numerical simulations to illustrate the accuracy of the proposed boundary conditions. We performed three types of simulation. The examples are inspired by works \cite{BNS}, \cite{BMN}. We show the different dispersive effects with a Gaussian and a wave packet initial data. We show also how to inject a travelling wave solution of \eqref{KN-inisyst} in the computational domain.

\section{Exact transparent boundary conditions}\label{EX}

In this section, we show how to derive transparent boundary conditions in the continuous case and prove the absorbing property of constructed conditions.

\subsection{Exact boundary conditions for linearised Green-Naghdi system}
We derive first the continuous boundary conditions for the system \eqref{KN-inisyst} of equations. We consider the initial value problem set on the whole space
\begin{equation*}
\begin{array}{r}
\eta_t + w_x = 0, 	\quad \forall x \in \mathbb{R}, \forall t > 0 \\
w_t  + \eta_x - \varepsilon w_{txx} = 0, 	\quad \forall x \in \mathbb{R}, \forall t > 0 \\
\eta(0,x) = \eta_0(x), \quad w(0,x) = w_0(x),\quad\forall x\in\mathbb{R} \\[2mm]
\lim\limits_{x\to\pm\infty} w(t,x) = \lim\limits_{x\to\pm\infty} \eta(t,x) = 0,
\end{array}
\end{equation*}
where the initial data $\eta_0$, $w_0$ are compactly supported functions in a finite interval $[x_\ell, x_r]$.  In order to construct the transparent boundary conditions, we  consider the solution of the problem set on the complementary of $[x_\ell, x_r]\subset\mathbb{R}$:
\begin{equation}\label{KN-compl}
\begin{array}{r}
\eta_t + w_x = 0, 	\quad \forall x \in \mathbb{R}\setminus [x_\ell, x_r], \forall t > 0 \\
w_t  + \eta_x - \varepsilon w_{txx} = 0, 	\quad \forall x \in \mathbb{R}\setminus [x_\ell, x_r], \forall t > 0 \\
\eta(0,x) = 0, \quad w(0,x) = 0,\quad \forall x\in\mathbb{R}\setminus[x_\ell, x_r] \\[2mm]

\lim\limits_{x\to\pm\infty} w(t,x) =\lim\limits_{x\to\pm\infty} \eta(t,x)  = 0,
\end{array}
\end{equation}

\noindent
This problem is homogeneous in time. We can apply the Laplace transform defined as
\begin{equation*}
\LL(f)(s;x) = \int\limits_{0}^{\infty}e^{-st} f(t;x)	dt
\end{equation*}
where $s$ is a parameter such as $\Re(s)>0$. (Hereafter $\Re$ denotes the real part), we obtain:
\begin{equation}\label{KN-Lapl}
\begin{array}{c}
s \LL(\eta)  + \partial_x\LL(w) = 0,\\
s \LL(w) + \partial_x\LL(\eta) -\varepsilon \partial_{xx} \LL(w) = 0. 
\end{array}
\end{equation}

\noindent
The solutions of the system \eqref{KN-Lapl} have the from
\begin{equation*}
\begin{array}{ll}
\left(
\begin{array}{c}
\LL(w)(s,x)\\
\LL(\eta)(s,x)
\end{array}
\right) = \alpha_+^r V^+ e^{\lambda^+ x} + \alpha_-^r V^- e^{\lambda^- x}, \quad \forall x>x_r, \\[5mm]
\left(
\begin{array}{c}
\LL(w)(s,x) \\
\LL(\eta)(s,x)
\end{array}
\right) = \alpha_+^\ell V^+ e^{\lambda^+ x} + \alpha_-^\ell V^- e^{\lambda^- x},\quad  \forall x<x_\ell
\end{array}
\end{equation*}
where $\alpha_+^{r,\ell}$, $\alpha_-^{r,\ell}$ are constant coefficients, $\lambda^+$, $\lambda^-$ are given by
\begin{equation*}\label{KN-lambs}
\dsp\lambda^+ = \sqrt[+]{\frac{s^2}{1 + \varepsilon s^2}}, 	\quad \lambda^- = -\sqrt[+]{\frac{s^2}{1 + \varepsilon s^2}},
\end{equation*}
and $V^+$, $V^-$ are the constant vectors:
 $$V^+ = (1, -\lambda^+/s)^T,\, V^- = (1, \lambda^+/s)^T.$$ 

The number $\sqrt[+]{z}$ corresponds to the principal square root of the complex number $z\in\mathbb{C}$. Note that the function $s \mapsto \sqrt{s^2/(1 + \varepsilon s^2)}$ maps $\Re(s) > 0$ to  $\mathbb{C} \setminus [0, 1/\varepsilon[$, therefore $\lambda^+$ has a strictly positive real part whereas $\lambda^-$ has a negative one. As a result, $x\mapsto e^{\lambda_+ x}$ increases exponentially fast as $x\to\infty$. In order to have a bounded solution $\LL(w)(s,x),\LL(\eta)(s,x)$ for all $x\geq x_r$, one must impose $\alpha_+^r=0$. Similarly, one has $\alpha_-^\ell=0$. The constant coefficients $\alpha_+^r$, $\alpha_-^\ell$ are written as:
\begin{equation*}\label{KN-htconds}
\begin{array}{c}
\dsp\alpha_+^r = \frac{1}{\sqrt[+]{1 + \varepsilon s^2}}\LL(w)(s,x_r) - \LL(\eta)(s,x_r) = 0, \\[5mm]	
\dsp\alpha_- ^\ell= \frac{1}{\sqrt[+]{1 + \varepsilon s^2}}\LL(w)(s,x_\ell) + \LL(\eta)(s,x_\ell) = 0.
\end{array}
\end{equation*}
\noindent
We then deduce a relation between  $\LL(\eta)$ and $\LL(w)$ at the boundary points $x_\ell, x_r$:
\begin{equation*}
\begin{array}{c}
\dsp\LL(w)(s,x_r)= \frac{1 + \varepsilon s^2}{\sqrt[+]{1 + \varepsilon s^2}} \LL(\eta)(s,x_r), 	\quad 
\dsp\LL(w)(s,x_\ell) = -\frac{1 + \varepsilon s^2}{\sqrt[+]{1 + \varepsilon s^2}} \LL(\eta)(s,x_\ell).
\end{array}
\end{equation*}
\noindent
The inversion of Laplace transform can be carried out explicitly and finally we get the following transparent boundary conditions:
\begin{equation}\label{KN-conds}
\begin{array}{c}
\dsp  w(t, x_r) = [\mathds{1} + \partial^2/\partial t]\frac{1}{\sqrt{\varepsilon}} \int_{0}^{t} \mathcal{J}_0( u/\sqrt{\varepsilon} ) \eta(t-u,x_r)du ,\\
\dsp w(t, x_\ell) = -[\mathds{1} + \partial^2/\partial t]\frac{1}{\sqrt{\varepsilon}} \int_{0}^{t} \mathcal{J}_0( u/\sqrt{\varepsilon} ) \eta(t-u,x_\ell)du,
\end{array}
\end{equation}
where $\mathcal{J}_0$  is the Bessel function of the first kind:
\begin{equation*}
\dsp \mathcal{J}_0(t) = \frac{1}{\pi}\int_{0}^{\pi} e^{i t \cos\theta}d \theta.
\end{equation*}

\noindent
Now we prove the following stability result.

\begin{proposition}\label{Absorb}
	The problem 
	
	\begin{equation}\label{KN-H1st}
	\begin{array}{r}
	\eta_t + w_x = 0, 	\quad \forall x \in ]x_\ell, x_r[, \forall t > 0 \\
	w_t  + \eta_x - \varepsilon w_{txx} = 0, 	\quad \forall x \in]x_\ell, x_r[, \forall t > 0 \\
	\eta(0,x) = \eta_0(x), \quad w(0,x) = w_0(x),\quad \forall x\in]x_\ell, x_r[ \\
	w(t,x_r) = \mathcal{L}^{-1}(\sqrt[+]{1+\varepsilon s^2})* \eta(t,x_r),\\
	w(t,x_\ell) = - \mathcal{L}^{-1}(\sqrt[+]{1+\varepsilon s^2})* \eta(t,x_\ell).
	\end{array}		
	\end{equation}
	
	is $L^\infty(\mathbb{R}^+, H^1(\mathbb{R}) \times L^2(\mathbb{R}))$ stable:  for all $t>0$ and for all smooth solution of \eqref{KN-H1st}, we have
	
	\begin{equation*}
	\dsp\int_{x_\ell}^{x_r}  \frac{\eta^2(t,x)}{2}  + \frac{w^2(t,x)}{2} + \frac{(\partial_x w(t,x))^2}{2} dx \leq \int_{x_\ell}^{x_r}  \frac{\eta_0^2(x)}{2}  + \frac{w_0^2(x)}{2} + \frac{(\partial_x w_0(t,x))^2}{2} dx.
	\end{equation*}		
\end{proposition}

\begin{proof}
	We determine directly from the equations a time-derivation of the generalised kinetic energy as 
	\begin{equation*}
	\frac{d}{dt}\int_{x_\ell}^{x_r}   \frac{\eta^2(t,x)}{2}  + \frac{w^2(t,x)}{2} + \frac{(\partial_x w(t,x))^2}{2}  dx = -[\eta(t,x) w(t,x)]_{x_\ell}^{x_r} + \varepsilon [w(t,x) \partial_{tx}w(t,x)]_{x_\ell}^{x_r},
	\end{equation*}
	where the brackets denote a jump of the function between $x_\ell$ and $x_r$. By integrating with respect to the time variable on the interval $(0,t)$, one obtains:
	\begin{multline*}
	\int_{x_\ell}^{x_r} \frac{\eta^2(t,x)}{2}  + \frac{w^2(t,x)}{2} + \frac{(\partial_x w(t,x))^2}{2} dx -  \int_{x_\ell}^{x_r} \frac{\eta_0^2(x)}{2}  + \frac{w_0^2(x)}{2} +  \frac{(\partial_x w_0(t,x))^2}{2} dx\\
	= \int_{0}^{t} (- \eta w +  \varepsilon w\partial_{tx}w)(\cdot,x_r) dt + \int_{0}^{t} (- \eta w +  \varepsilon w\partial_{tx}w)(\cdot,x_\ell)dt := J_r - J_\ell
	\end{multline*}
	if $J_r \leq 0$ and  $J_\ell \geq 0$ then the inequality is satisfied. Let us first consider the right value $J_r$, we fix the $T>0$ and denote $N = \eta(t, x_r)\cdot 1_{[0,T]}$, $W = w(t, x_r)\cdot 1_{[0,T]}$. Note that from the first equation of \eqref{KN-H1st}, one deduces  that  $\partial_{xt} W = \partial_{tt} N$ and obtains
	{\setlength\arraycolsep{1pt}
	\begin{eqnarray*}
	J_r &:=&  \int_{0}^{t} (- \eta w +  \varepsilon w\partial_{tx}w)(\cdot,x_r) dt  =  \int_{-\infty}^{\infty} (- N W +  \varepsilon W N'') dt = \\
	&=&  \int_{-\infty}^{\infty} \left(- N \left(Op(\sqrt[+]{1+\varepsilon s^2}) * N\right) +  \varepsilon\, \left(Op(\sqrt[+]{1+\varepsilon s^2}) * N\right) N''\right) dt \\ 
	&=& \frac{1}{2\pi} \Re \int_{-\infty}^{\infty} \left( -\overline{\hat{N}} \sqrt[+]{1 - \varepsilon \xi^2} \hat{N} + \varepsilon \sqrt[+]{1 - \varepsilon \xi^2} \hat{N} (-\xi^2) \overline{\hat{N}}  \right) d\xi\\
	&=& -\frac{1}{2\pi} \Re \int_{-\infty}^{\infty} (1 + \varepsilon \xi^2)  \sqrt[+]{1 - \varepsilon \xi^2} \vert \hat{N} \vert^2 d\xi \leq 0.
	\end{eqnarray*}
	}
\noindent
Here $\hat N $ denotes the Fourier transform of $N$. When $\xi$ is smaller than $\varepsilon^{-1/2}$, the real part of the integral has the positive value, as the square root is real. On the other hand, if $\xi > \varepsilon^{-1/2}$, then the square root is pure imaginary and the real part is identically equal to zero. An estimate for $J_\ell$ can be done similarly:
	\begin{equation*}
	J_\ell = \frac{1}{2\pi} \Re \int_{-\infty}^{\infty} (1 + \varepsilon \xi^2)  \sqrt[+]{1 - \varepsilon \xi^2} \vert \hat{N} \vert^2 d\xi \geq 0.
	\end{equation*}
	\noindent
	This completes the proof of the proposition.
\end{proof}

\subsection{Exact boundary conditions for the linear Boussinesq equation}

The system \eqref{KN-inisyst} is equivalent to the linearized Boussinesq equation:
\begin{equation}\label{KN-oneeq}
	(w - \varepsilon w_{xx} )_{tt} - w_{xx} = 0, 	\quad \forall x \in \mathbb{R}, \forall t > 0.
\end{equation}
The continuous boundary conditions for equation \eqref{KN-oneeq} are required for the Crank Nicolson scheme on a staggered grid. We consider the initial value problem set on the whole space
\begin{equation*}\label{KN-wholesp1eq}
	\begin{array}{r}
	(w - \varepsilon w_{xx} )_{tt} - w_{xx} = 0, \quad \forall x \in \mathbb{R}, \forall t > 0 \\
	w(0,x) = w_0(x), \quad
	w_t(0,x) = v_0(x), \quad \forall x\in\mathbb{R} \\[2mm]
	
	\lim\limits_{x\to\infty} w(t,x) = \lim\limits_{x\to-\infty} w(t,x)  = 0,
	\end{array}
\end{equation*}
where the initial data $w_0$, $v_0$ are compactly supported in $[x_\ell, x_r]$. The problem set on the complementary of $[x_\ell, x_r]\subset\mathbb{R}$ reads as:
\begin{equation*}\label{KN-compl1eq}
\begin{array}{r}
	(w - \varepsilon w_{xx} )_{tt} - w_{xx} = 0, 	\quad \forall x \in \mathbb{R}, \forall t > 0\\[2mm]
	w(0,x) = 0,\quad w_t(0,x) = 0, \quad\forall x\in\mathbb{R}\setminus[x_\ell, x_r]\\[2mm]
\lim\limits_{x\to\infty} w(t,x) = \lim\limits_{x\to-\infty} w(t,x)  = 0.
\end{array}
\end{equation*}
\noindent
By applying the Laplace transform,  one finds:
\begin{equation*}
		s^2 \left(\LL(w)(s,x) - \varepsilon \partial_{xx} \LL(w)(s,x) \right)-  \partial_{xx}\LL(w)(s,x) = 0. 
\end{equation*}

We are searching again for the solution decreasing at infinity, so that give us one condition on the left boundary and one on the right one for the function $\LL(w)$:

\begin{equation*}
	\dsp\partial_x\LL(w)(s,x_r)  = -\frac{s}{\sqrt[+]{1 + \varepsilon s^2}} \LL(w)(s,x_r), \quad
	\dsp\partial_x\LL(w)(s,x_\ell)  = \frac{s}{\sqrt[+]{1 + \varepsilon s^2}} \LL(w)(s,x_\ell) .
\end{equation*}

 The inversion of Laplace transform can be found explicitly and finally we get
 \begin{equation}\label{KN-conds1eq}
 \begin{array}{c}
	\dsp  w_x(t, x_r) = -\partial/\partial t \int_{0}^{t} \mathcal{J}_0( u/\sqrt{\varepsilon} ) w(t-u,x_r)du ,\\[4mm]
	\dsp  w_x(t, x_\ell) =   \partial/\partial t \int_{0}^{t} \mathcal{J}_0( u/\sqrt{\varepsilon} ) w(t-u,x_\ell)du.
 \end{array}
 \end{equation}

For these boundary conditions, the absorbing property is fulfilled as well:
\begin{proposition}\label{Absorb1eq}
	Any smooth solution of the problem 
	
	\begin{equation}\label{KN-H1st1eq}
		\begin{array}{r}
			(w - \varepsilon w_{txx} )_{tt} - w_{xx} = 0, 	\quad \forall x \in [x_l,x_r], \forall t > 0 \\
			w(0,x) = w_0(x), \quad w_t(0,x) = v_0(x),\quad \forall x\in]x_\ell, x_r[\\
			\dsp  w_x(t, x_r) = -\partial_t \int_{0}^{t} \mathcal{J}_0( s/\sqrt{\varepsilon} ) w(t-s,x_r)ds ,\\[4mm]
			\dsp  w_x(t, x_\ell) =   \partial_t \int_{0}^{t} \mathcal{J}_0( s/\sqrt{\varepsilon} ) w(t-s,x_\ell)ds.
		\end{array}		
		\end{equation}
		
		satisfies for all $t>0$ the following estimate:	
		\begin{equation*}
			 \dsp\int_{x_\ell}^{x_r}  \left( \frac{(\partial_t w)^2}{2}   + (\partial_{x} w)^2 + \varepsilon (\partial_{tx} w)^2\right)(t,x) dx \leq \int_{x_\ell}^{x_r} \left( \frac{(\partial_t w_0)^2}{2}  + (\partial_{x} w_0)^2 + \varepsilon (\partial_{tx} w_0)^2\right)dx.
		\end{equation*}		
\end{proposition}

 The discretization of the conditions \eqref{KN-conds} or the conditions \eqref{KN-conds1eq} is not trivial task. In the next section, we show how to obtain a consistent discretization of the boundary conditions which is compatible with the discrete numerical scheme used to carry out simulation of the model \eqref{KN-inisyst}.  The proofs of consistency with the continuous conditionsare carried out in the sections \ref{Disc1eq}, \ref{Disc}.

\section{Discrete transparent boundary conditions: Staggered grid}\label{Disc1eq}

 In this section we derive discrete artificial boundary conditions for the linearized Green-Naghdi system \eqref{KN-inisyst}. In order to build up these conditions, we follow  the strategy found in \cite{BEL-V} and \cite{BMN} and consider directly the problem on the fully discretized equations. In this section, we focus on spatial discretization on a staggered grid and Crank Nicolson time discretization. The numerical scheme is written as:

 \begin{equation}\label{KN-CrNicDis}
 \begin{array}{c}
	 \dsp\frac{\eta_{j+1/2}^{n+1} - \eta_{j+1/2}^{n}}{\delta t} + \frac{1}{2} \left( \frac{w_{j+1}^{n+1} - w_{j}^{n+1}}{ \delta x} +  \frac{w_{j+1}^{n} - w_{j}^{n}}{ \delta x} \right) = 0,\\[4mm]
	 \dsp\frac{w_j^{n+1} - w_j^{n}}{\delta t} - \frac{\varepsilon}{\dt} \left(  \frac{w_{j+1}^{n+1} - 2 w_{j}^{n+1} + w_{j-1}^{n+1}}{\delta x^2} -  \frac{w_{j+1}^{n} - 2  w_{j}^{n}+ w_{j-1}^{n}}{\delta x^2} \right) 
			  \\[4mm]
			 \dsp +\frac{1}{2} \left( \frac{\eta_{j+1/2}^{n+1} - \eta_{j-1/2}^{n+1}}{\delta x} +  \frac{\eta_{j+1/2}^{n} - \eta_{j-1/2}^{n}}{\delta x} \right) = 0,\\[4mm]
			 1 \leq j \leq J, n \in \mathbb{N}
 \end{array}
 \end{equation}
where $\delta t>0$, $\delta x>0$ are time and space step, respectively, and number of space cells $J\in\mathbb{N}$ is calculated as
\begin{equation*}
	J = \frac{x_r - x_\ell}{\delta x}.
\end{equation*}
A staggered grid is a setting for the spatial discretization, in which the unknown are not evaluated at the same space position. That is to say $w^n_j \approx w(n\delta t,x_\ell+j \delta x)$, $\eta^n_{j+1/2} \approx \eta(n\delta t,x_\ell+(j + 1/2) \delta x)$.


The procedure mimic what was done for the continuous case in the previous section. We first apply a discrete analogue of the Laplace transform which is referred to as $\mathcal{Z}-$transform. The definition reads as follows:
\begin{equation*}
	\widehat{u}(z) =  \mathcal{Z}\{ (u)_n \} (z)  =  \sum\limits_{n\geq 0} u_n z^{-n}, \quad \vert z \vert > R > 0,
\end{equation*}
 $z$ is the complex variable and $R$ is the radius of convergence of Laurent series. Hereafter the hat will denote the result of $\mathcal{Z}-$transform of the discrete sequences $\eta^n_{j+1/2}$, $w^n_j$ with respect to time index $n$.  
 
 The discrete system \eqref{KN-CrNicDis} reduces to the linear recurrence relations:
 \begin{equation} \label{KN-Ztr} 
 \begin{array}{c}
 \dsp \widehat{\eta}_{j+1/2} = -\frac{1}{s(z) \dx}(\widehat{w}_{j+1} - \widehat{w}_{j}),
 \\[4mm]
 \dsp - \frac{\varepsilon s(z)}{\dx^2} \widehat{w}_{j-1} + s(z)\left(1 + \frac{2 \varepsilon}{\dx^2}\right) \widehat{w}_{j} - \frac{\varepsilon s(z)}{\dx^2} \widehat{w}_{j+1} + \frac{\widehat{\eta}_{j+1/2} - \widehat{\eta}_{j-1/2}} {\dx} = 0,
 \end{array},
 \end{equation}
 where,
 \begin{equation}\label{KN-defs}
	 \dsp s(z) =  \frac{2}{\delta t} \frac{z - 1}{z + 1}.
 \end{equation}
 \noindent
As the function $z\mapsto s(z)$ has a singularity at $z = -1$, we assume $\vert z \vert > 1$, which in turn yields $\Re\left(s(z)\right) > 0$. 
Note that the initial values $w^0_j$, $\eta^0_j$ are supposed to be zero for all $j\leq 0$, $j \geq J+1$.\\
 
We can eliminate $\hat\eta_{j+1/2}$ from the system  \eqref{KN-Ztr}  so as to obtain a scalar recurrence relation: 
 \begin{equation}\label{KN-EqWDis}
 \dsp (1 + \varepsilon s^2(z)) \widehat{w}_{j-1} - 2\left(1+s(z)\left(\varepsilon + \frac{\dx^2}{2}\right)\right) \widehat{w}_{j} + (1 + \varepsilon s^2(z)) \widehat{w}_{j+1}  = 0,\quad
 1 \leq j \leq J, n \in \mathbb{N}
 \end{equation}
\noindent
 This linear recurrence has a general solution written in the form: 
 \begin{equation*}
	\widehat{w}_j = \alpha_+ r_+(z)^j + \alpha_- r_-(z)^j,\quad j\in\mathbb{Z}
 \end{equation*}
 where $r_{\pm}$ are the roots of characteristic polynomial $P$ associated with the recurrence:
  \begin{equation}\label{KN-poly1eq}
			P(r) = (1 + \varepsilon s^2(z)) r^2 - 2\left(1+s^2(z)\left(\varepsilon + \frac{\dx^2}{2}\right)\right) r + (1 + \varepsilon s^2(z)).
  \end{equation}
 The explicit formulae for the roots reads
 \begin{equation}\label{KN-roots}
 r_\pm(z) = 1 + \frac{ s^2(z)\dx^2}{2 (1 + \varepsilon s^2(z))} \pm \frac{s(z) \dx  \sqrt{\dx^2 + 4(1 + \varepsilon s^2(z))}}{2 (1 + \varepsilon s^2(z))}.
 \end{equation}
 We show now an important property of the roots \eqref{KN-roots}:
 \begin{proposition}\label{KN-PrSepar1eq}
 	The roots of characteristic polynomial \eqref{KN-poly1eq} associated with linear recurrence relation have the following separation property: for all $z\in\mathbb{C}$ such that $|z|>1$, one has 
 	\begin{equation*}
		\vert r_+(z) \vert > 1, \quad  \vert r_-(z) \vert < 1.
 	\end{equation*}
 \end{proposition}
 
 \begin{proof}
First let us show that there is no root on the unit circle. We assume that there is a root $r = e^{i\phi}$ such that $P(r)=0$. This equation reads

\begin{equation*}
  (1 + \varepsilon s^2(z)) e^{2i\phi} - 2\left(1+s^2(z)\left(\varepsilon + \frac{\dx^2}{2}\right)\right) e^{i\phi} + (1 + \varepsilon s^2(z)) = 0
\end{equation*}
and one deduces that
 \begin{equation*}
 s^2(z) = -\frac{4 \sin^2\phi}{2\varepsilon(1-\cos(\phi)) + \delta x^2} \in\mathbb{R}^-,
 \end{equation*}
and therefore $\Re(s) = 0$, which is in contradiction with the assumption $\vert z \vert > 1$. Therefore, there is no root of $P$ on the unit circle.

The product of the roots is equal to one due to relation between the coefficients of $P$ and there are no roots with modulus one. Therefore there is necessarily one root with a modulus larger than one and the other one with modulus smaller than one. In the limit $\vert s(z) \vert \rightarrow \infty$ one has $\vert r_+(z) \vert > 1$ and $\vert r_-(z) \vert < 1$. By continuity of $z\mapsto|r_\pm(z)|$ on the domain $\{z\in\mathbb{C},\:|z|>1\}$, this remains true for all $|z|>1$. This completes the proof of the proposition.
\end{proof}

 
 The construction of the boundary conditions is then carried out just as in the continuous case. First note that the solution to \eqref{KN-EqWDis} reads
 $$
 \begin{array}{ll}
\displaystyle
\widehat w_j=\alpha_+^r r_+(z)^j+\alpha_-^r r_-(z)^j,\quad \forall j\geq J,\\[2mm]
\displaystyle
\widehat w_j=\alpha_+^\ell r_+(z)^j+\alpha_-^\ell r_-(z)^j,\quad \forall j\leq 1.
\end{array}
 $$
 \noindent
 We search for bounded solutions,  which means that $\alpha_-^\ell = 0$, and   $\alpha_+^r = 0$. 
%
%
 These conditions are equivalent to the boundary conditions:
 
 \begin{equation}\label{KN-cond}
		\widehat{w}_1 = r_+(z) \widehat{w}_0, \quad \widehat{w}_{J+1} = r_-(z) \widehat{w}_J.
 \end{equation}
 Here we have the conditions for the images in the \ZT domain. In order to apply the \ZT inverse transform, we present the conditions \eqref{KN-cond} in the following form (we have used the explicit formula for $s(z)$):
\begin{equation}\label{bclr}
\begin{array}{l}
	\dsp\widehat{w}_1 = \left(1 + \frac{ 2\dx^2 (z-1)^2}{\Lambda z^2 - 2 \mu z + \Lambda} + \frac{2\dx(z-1) \sqrt{\Gamma z^2 - 2 \nu z + \Gamma}}{\Lambda z^2 - 2 \mu z + \Lambda}\right) \widehat{w}_0, \\[4mm]
	\dsp\widehat{w}_{J+1} =  \left(1 + \frac{ 2\dx^2 (z-1)^2}{\Lambda z^2 - 2 \mu z + \Lambda} - \frac{2\dx(z-1)\sqrt{\Gamma z^2 - 2 \nu z + \Gamma}}{\Lambda z^2 - 2 \mu z + \Lambda}\right) \widehat{w}_J,
\end{array}
\end{equation}
 where $\Lambda = 4\varepsilon + \dt^2$, $\mu = 4\varepsilon - \dt^2$, $\Gamma = \Lambda + \dx^2$, $\nu = \mu + \dx^2$. The inversion of constructed conditions can be done explicitly and it is a {\em key aspect} in using the scheme on a staggered grid. In the next section we will show that such inversion is not possible for scheme with collocated grid, and an other strategy for inversion should be used.
 
We focus on the inversion of the left boundary condition, the treatment of the right one being similar. Let us first to mention a useful result  for the inversion of \eqref{bclr}, namely
\begin{lemma}
	\begin{equation*}
		\mathcal{Z}^{-1} \left( \frac{z}{\sqrt{ z^2 - 2 \nu z + 1}} \right) = \sum\limits_{n=0}^{\infty} \PP_n(\nu) z^{-n},
	\end{equation*}
	for all $\vert z \vert > max(z_1,z_2)$, where $z_1$, $z_2$ are the roots of $z^2 - 2 \nu z + 1$ and $\PP_n(\nu)$ is the $n-$th Legendre polynomial.
\end{lemma}

In order to use this result, we write
\begin{equation*}
	\sqrt{\Gamma z^2 - 2 \nu z + \Gamma} = \sqrt{\Gamma} (z - 2 v + z^{-1}) \frac{z}{\sqrt{z^2 - 2 v z +1}},\quad v=\frac{\nu}{\Gamma}.
\end{equation*}
\noindent
By multiplying the left boundary condition by $\Lambda z^2 - 2 \mu z + \Lambda$ and by using the inverse shift property of \ZT transform, one finds
\begin{multline}\label{KN-left}
\Lambda w_1^{n+1} - (\Lambda + 2\dx^2 + 2 \dx \sqrt{\Gamma})w_0^{n+1} = 2(\mu w_1^{n} - (\mu + 2 \dx^2 + \dx \sqrt{\Gamma} (v + 1))w_0^{n}) -\\ - (\Lambda w_1^{n-1} - (\Lambda + 2\dx^2 + 2 \dx \sqrt{\Gamma})w_0^{n-1}) +
		 2 \dx\sqrt{\Gamma} \left( (\PP_2 - 2 v^2 + v)  w_0^{n-1} + \sum_{k =2}^{n} s_k  w_0^{n-k} \right).
\end{multline}
A similar calculation gives the boundary condition on the right:
\begin{multline}\label{KN-right}
\Lambda w_{J+1}^{n+1} - (\Lambda + 2\dx^2 - 2 \dx \sqrt{\Gamma})w_{J}^{n+1} = 2(\mu w_{J+1}^{n} - (\mu + 2 \dx^2 - \dx \sqrt{\Gamma} (v + 1))w_J^{n}) -\\ - (\Lambda w_{J+1}^{n-1} - (\Lambda + 2\dx^2 - 2 \dx \sqrt{\Gamma})w_J^{n-1}) - 2 \dx\sqrt{\Gamma} \left( (\PP_2 - 2 v^2 + v)  w_J^{n-1} + \sum_{k =2}^{n} s_k(v)  w_J^{n-k} \right)
\end{multline}
\noindent
where 
\begin{equation*}
	\forall k\in\N \quad  s_k(v) = \PP_{k+1}(v) -(2v + 1) \PP_k(v)  + (2v + 1)\PP_{k-1}(v)  - \PP_{k-2}(v),
\end{equation*}
and $\PP_{-1}=0$, $\PP_{-2}=0$. As a conclusion, the full scheme consists in boundary conditions (\ref{KN-left}) and (\ref{KN-right}) together with the interior scheme written as

 \begin{multline}\label{KN-Disc1eq}
 -a_+ w_{j+1}^{n+1}  + (1 + 2a_+) w_{j}^{n+1} - a_+ w_{j-1}^{n+1} = 2 (-a_- w_{j+1}^{n}  + (1 + 2a_-) w_{j}^{n} - a_- w_{j-1}^{n}  )\\
 - (-a_+ w_{j+1}^{n-1}  + (1 + 2a_+) w_{j}^{n-1} - a_+ w_{j-1}^{n-1}), \quad  1 \leq j \leq J, n\in\N,
 \end{multline}
 where
 \begin{equation*}
 a_- = \frac{\varepsilon - \dt/4}{\dx^2}, \quad a_+ = \frac{\varepsilon + \dt/4}{\dx^2}.
 \end{equation*}

The interior scheme \eqref{KN-Disc1eq} is second order accurate in time and in space. In what follows, we check the consistency of the boundary conditions \eqref{KN-left} and \eqref{KN-right}.
 
 \subsection{Consistency theorem}

 In order to provide a good approximation of the continuous solution of \eqref{KN-oneeq} by numerical solution of \eqref{KN-Disc1eq} with  \eqref{KN-left}, \eqref{KN-right}, one should prove a consistency result. In what follows, we show that \eqref{KN-left} and \eqref{KN-right} are second order accurate in time and space.
 
 \begin{theorem}\label{ConsistStag}
	 Let $w$ be a smooth solution \eqref{KN-oneeq} which satisfies the transparent boundary conditions \eqref{KN-conds1eq}. We define the \ZT transform of $w(\cdot,x)$ for all $x\in[x_\ell,x_r]$ by
	\begin{equation*} 
		\forall z \neq 0, \quad \widehat{w}(z,x) = \sum\limits_{n = 0}^{\infty} \frac{w(n\dt, x)}{z^{n}}.
	\end{equation*}
	Then for all compact $K\in\mathbb{C}^+=\{z\in\mathbb{C}, \:\Re(z)> 0\}$, for all $s \in K$, we have
	 \begin{equation*}
	 \begin{array}{ll}
	 	\hat{w}(e^{s\dt},x_\ell+\dx) - r_+(e^{s\dt})\hat{w}(e^{s\dt},x_\ell)  =  O(\dt^2 + \dx^2),\\[2mm]
		 \hat{w}(e^{s\dt},x_r) - r_-(e^{s\dt})\hat{w}(e^{s\dt}, x_r-\dx)  =  O(\dt^2 + \dx^2),
		 \end{array}
	 \end{equation*}
	 where $r_\pm(z)$ defined by \eqref{KN-roots}.
 \end{theorem}
 
 \begin{proof}

First of all, let us note that the \ZT transform, defined above, is an approximation of the Laplace transform. More precisely,  for all smooth functions $f(0)=f'(0)=\dots =f^{(k)}(0)= 0$ ($k\in\mathbb{N}$), and all $s\in\mathbb{C}^+$,  we have:
\begin{equation}\label{LApprox}
  \LL(f)(s) = \dt \hat{f}(e^{s\dt}) + O(\dt^{k+2}),
\end{equation}
where $s$ is a parameter of Laplace transform. See \cite{BMN} for a proof of this result. Recalling definition of the roots \eqref{KN-roots}, we have
  \begin{equation*}
  \begin{array}{c}
  		\dsp\widehat{w}(z,x_\ell+\dx) - r_+(z)\widehat{w}(z,x_\ell)  =\\[5mm] \dsp	 \widehat{w}(z,x_\ell+\dx) -\left(1 + \frac{ s^2(z)\dx^2}{2 (1 + \varepsilon s^2(z))} + \frac{s(z)\dx \sqrt{\dx^2 + 4(1 + \varepsilon s^2(z))}}{2 (1 + \varepsilon s^2(z))}\right)\widehat{w}(z,x_\ell) = \\[5mm]  \dsp \widehat{w}(z,x_\ell+\dx) -  \widehat{w}(z,x_\ell)  - \frac{s^2(z) \dx^2}{2(1+\varepsilon s^2(z))}\widehat{w}(z,x_\ell)  -   \frac{s(z)\dx}{\sqrt{1 + \varepsilon s^2(z)}} \sqrt{1 + \frac{s^2(z) \dx^2}{4(1 + \varepsilon s^2(z))}}\widehat{w}(z,x_\ell).		
  \end{array}
  \end{equation*}
Note that the function $s(z)$ defined in \eqref{KN-defs} with $z = e^{s\dt}$ is approximated as
\begin{equation*}
s(z) = s + O(\dt^2).
\end{equation*}
We then find
    \begin{equation*}
  \begin{array}{c}
 \dsp\widehat{w}(z,x_\ell+\dx) - r_+(z)\widehat{w}(z,x_\ell)  =\\
  \dsp \dx \left(\partial_x\widehat{w}(e^{s\dt},x_\ell) + \frac{\dx}{2} \partial_{xx}\widehat{w}(e^{s\dt},x_\ell)  - \frac{s^2 \dx}{2(1+\varepsilon s^2)}\widehat{w}(e^{s\dt},x_\ell)  -   \frac{s}{\sqrt{1 + \varepsilon s^2(z)}} \widehat{w}(e^{s\dt},x_\ell)\right) + O(\dx^2) \\[5mm] 
  = \dsp \frac{\dx}{\dt} \left(\dt\left( \partial_x \widehat{w}(e^{s\dt}, x_\ell) - \frac{s}{\sqrt{1 + \varepsilon s^2(z)}} \widehat{w}(e^{s\dt},x_\ell) \right)  +\dt \frac{\dx}{2} \left( \partial_{xx}\widehat{w}(e^{s\dt},x_\ell)  - \frac{s^2}{1+\varepsilon s^2}\widehat{w}(e^{s\dt},x_\ell)  \right)  \right) + O(\dx^2)
    \end{array}
  \end{equation*}

By applying relation \eqref{LApprox} to the last line, we find that the first expression between  the parentheses is the Laplace transform of the continuous boundary condition on the left and the second one is the Laplace transform of \eqref{KN-oneeq}:
    \begin{equation*}
\begin{array}{c}
	\dsp\widehat{w}(z,x_\ell+\dx) - r_+(z)\widehat{w}(z,x_\ell)  = \\[5mm] \dsp \dx \left(\partial_x \LL(w) - \frac{s}{\sqrt{1 + \varepsilon s^2}}\LL(w) \right)  
+\dsp\frac{\dx^2}{2} \left(\partial_{xx} \LL(w) - \frac{s^2}{1 + \varepsilon s^2}  \LL(w)\right) + O(\dx^2 + \dt^2)=O(\dx^2 + \dt^2).
    \end{array}
\end{equation*}
This completes the proof of consistency for the boundary on the left, the proof being similar for the boundary on the right. 
 \end{proof}
 
 In addition to consistency results, we show that the coefficients $s_k$ of boundary conditions show a stable behavior. It can be demonstrated numerically (see figure, \ref{coef}, a). Moreover the coefficients decrease as $n^{-3/2}$ just as for BBM equation \cite{BMN}, mixed BBM-KDV equation \cite{BNS} or  Schr\"{o}dinger equation \cite{ET}. The boundary conditions are then non sensitive to round off errors on the numerical solution.
 
 \begin{figure}[t]
 	\begin{center}
 		\begin{minipage}[h]{0.495\linewidth}
 			\center{ \includegraphics[width=1\linewidth]{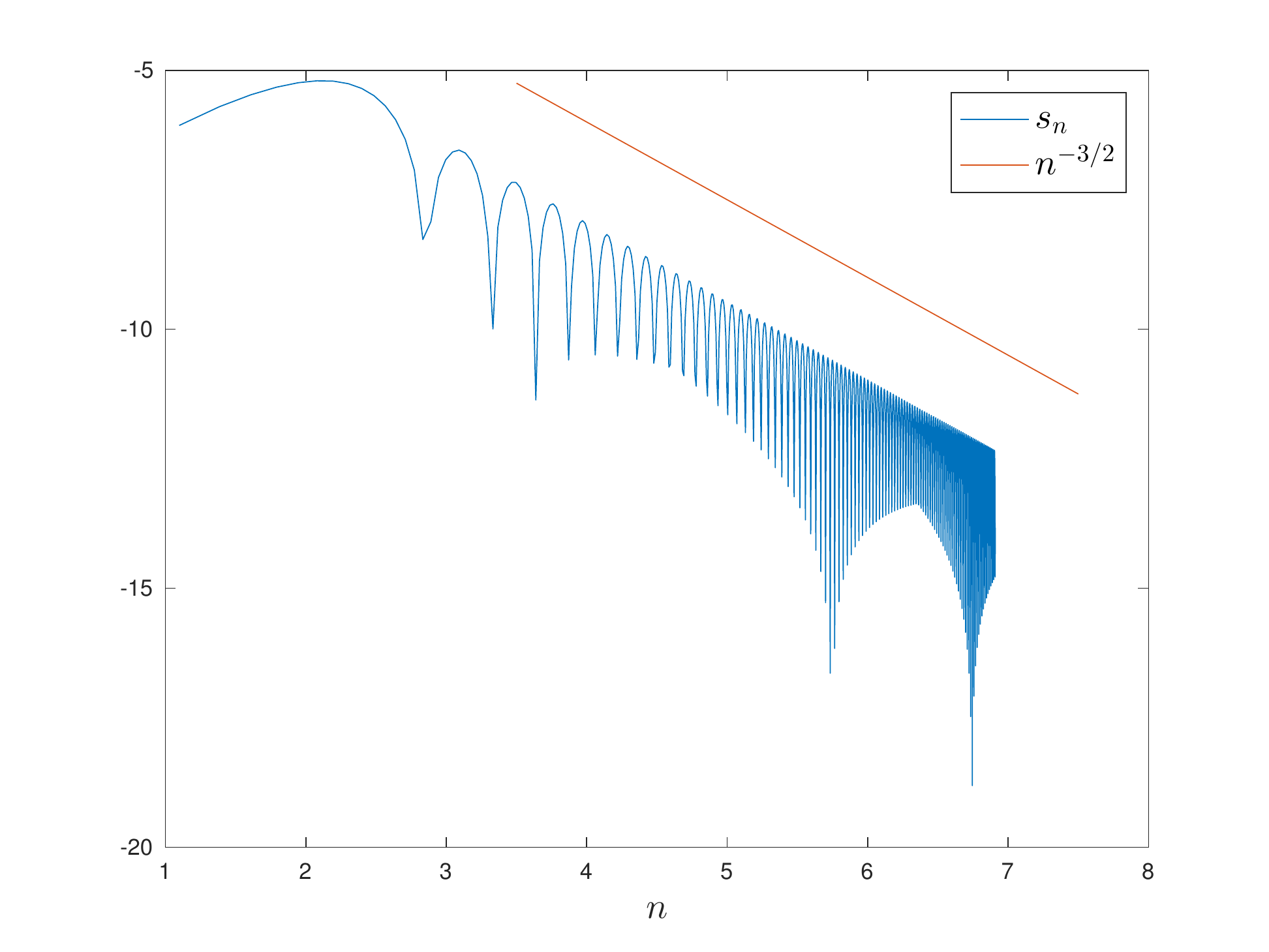} \\ {\it a}}
 		\end{minipage}
 		\hfill
 		\begin{minipage}[h]{0.495\linewidth}
 			\center{\includegraphics[width=1\linewidth]{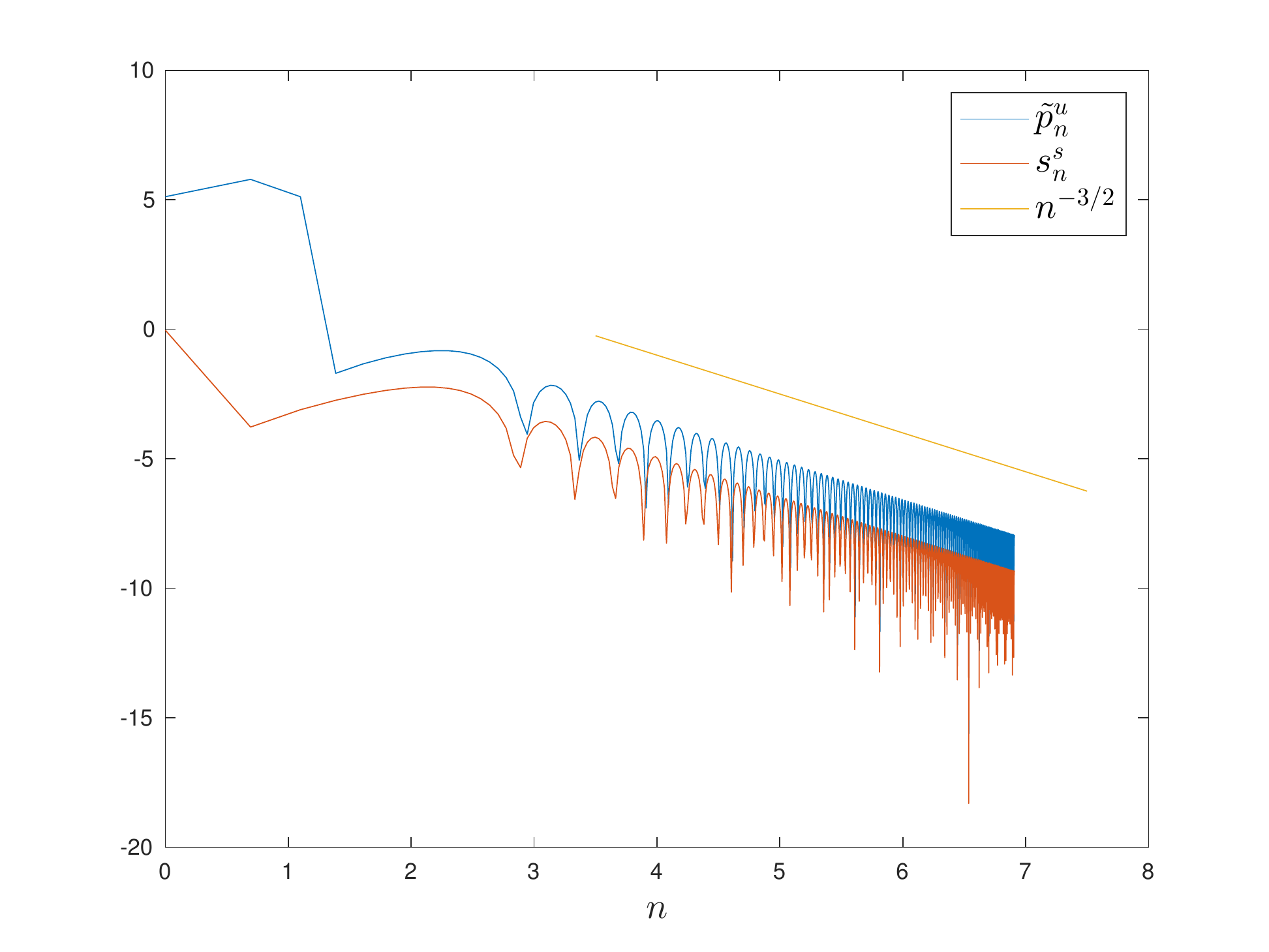} \\ {\it b}}\\
 		\end{minipage}
 	\end{center}
 	\caption{Coefficients of the discrete boundary conditions \eqref{KN-left}, \eqref{KN-right} (a) and \eqref{KN-BCleft}, \eqref{KN-BCright} (b) with $\delta x = 2^{-10}$, $\delta t = 10^{-2}$, $\varepsilon = 10^{-3}$}
 	\label{coef}
 \end{figure}
 
 \section{Discrete transparent boundary conditions: Collocated grid}\label{Disc}
 
 In this section, we consider transparent boundary conditions associated to a spatial discretization of \eqref{KN-inisyst} on collocated grids (functions $\eta$, $w$ are evaluated at the same points) . We keep a Crank Nicolson time discretization. The numerical scheme reads as follow:

{\setlength\arraycolsep{1pt}
 \begin{eqnarray}
 \dsp\frac{\eta_j^{n+1} - \eta_j^{n}}{\delta t} &+& \frac{1}{2} \left( \frac{w_{j+1}^{n+1} - w_{j-1}^{n+1}}{2 \delta x} +  \frac{w_{j+1}^{n} - w_{j-1}^{n}}{2 \delta x} \right) = 0,\nonumber\\
 \label{KN-CrNicDisCol}
 \dsp\frac{w_j^{n+1} - w_j^{n}}{\delta t}&-& \varepsilon \left(  \frac{w_{j+1}^{n+1} - 2 w_{j}^{n+1} + w_{j-1}^{n+1}}{\delta x^2} -  \frac{w_{j+1}^{n} - 2  w_{j}^{n}+ w_{j-1}^{n}}{\delta x^2} \right) \\
\dsp
& +& \frac{1}{2} \left( \frac{\eta_{j+1}^{n+1} - \eta_{j-1}^{n+1}}{2 \delta x} +  \frac{\eta_{j+1}^{n} - \eta_{j-1}^{n}}{2 \delta x} \right) = 0,\nonumber
 \end{eqnarray}}
\noindent
for all $1\leq j\leq J$ and $n\in\mathbb{N}$. By applying \ZT transform, the system \eqref{KN-CrNicDisCol} reduces to the second order linear recurrence system ($\vert z \vert > 1$):
\begin{equation}\label{ns-colloc}
\begin{array}{r}
\widehat{w}_{j+1} = \widehat{w}_{j-1} - s(z) \widehat{\eta}_j,
\\[2mm]
\dsp\widehat{\eta}_{j+1} = \widehat{\eta}_{j-1} + \frac{\varepsilon\,s(z)}{\delta x^2} \widehat{w}_{j+1} - s(z) (1 + \frac{2\varepsilon}{\delta x^2})\widehat{w}_j + \frac{\varepsilon\, s(z)}{\delta x^2}  \widehat{w}_{j-1},
\end{array}, \quad \dsp s(z) = \frac{2}{\delta t} \frac{z -1}{z+ 1};
\end{equation}
 \noindent
 We search for a basis of solutions of this recurrence system. We first write \eqref{ns-colloc} as a first order recurrence system
 \begin{equation*}
 \left(
 \begin{array}{c}
 \widehat{w}_{j+1}\\
 \widehat{\eta}_{j+1}\\
 \widehat{t}_{j+1}\\
 \widehat{r}_{j+1}
 \end{array}
 \right)= 
 \left(
 \begin{matrix}
 0 & -2 \delta x s(z) & 1 & 0\\
 - 2 \delta x s(z) (1 +2\varepsilon/\delta x^2) & -4 \varepsilon s^2(z) & 4\varepsilon s(z)/\delta x & 1\\
 1 & 0 & 0 & 0\\
 0 & 1 & 0 & 0
 \end{matrix}
 \right) \left(
 \begin{array}{c}
 \widehat{w}_{j}\\
 \widehat{\eta}_{j}\\
 \widehat{t}_{j}\\
 \widehat{v}_{j}
 \end{array}
 \right):=A(z)\left(
 \begin{array}{c}
 \widehat{w}_{j}\\
 \widehat{\eta}_{j}\\
 \widehat{t}_{j}\\
 \widehat{v}_{j}
 \end{array}
 \right),
 \end{equation*}
 where we have set $\hat{t}_{j}= \hat{w}_{j-1}$, $\hat{v}_{j}= \hat{\eta}_{j-1}$. The solutions of this recurrence system have the form
 \begin{equation*}\label{KN-sol}
 \left(\widehat{w}_j, \widehat{\eta}_j,  \widehat{t}_j,  \widehat{v}_j\right)^T = \sum\limits_{k=1}^{4} \alpha_k^r r_k^j V_k, \quad \forall j\geq J+1,\quad 
  \left(\widehat{w}_j, \widehat{\eta}_j,  \widehat{t}_j,  \widehat{v}_j\right)^T = \sum\limits_{k=1}^{4} \alpha_k^\ell r_k^j V_k, \quad \forall j\leq 0
 \end{equation*}
 where $r_k$, $k = 1, 2, 3, 4$ are the roots of characteristic polynomial $P$ associated to the matrix $A(z)$
 \begin{equation}\label{KN-poly}
 P(r) = r^4 + 4\varepsilon s^2(z) r^3 - \left(2 + 4 s^2(z) ( \delta x^2 + 2 \varepsilon)\right) r^2+  4\varepsilon s^2(z)  r + 1,
 \end{equation}
 whereas $V_k$ are the corresponding eigenvectors, and $\alpha_k^{r,\ell}$ are constant coefficients. 
 The expression for the roots of $P(r)$ are explicit but useless when we will have to carry out the inversion of the $\mathcal{Z}-$transform. Though, we can prove the following property.
 
 \begin{proposition}\label{KN-PrSepar}
 	The roots of the characteristic polynomial $P$ given by \eqref{KN-poly} have the following separation property: for all $z\in\mathbb{C}$ such that $|z|>1$, one has
 	\begin{equation*}
 	\vert r_1(z) \vert > 1, \quad \vert r_2(z) \vert > 1, \quad \vert r_3(z) \vert < 1, \quad \vert r_4(z) \vert < 1.
 	\end{equation*}
 	Here the roots are ordered as $\vert r_1(z) \vert \geq  \vert r_2(z) \vert \geq  \vert r_3(z) \vert \geq  \vert r_4(z) \vert$.
 \end{proposition}
 
 \begin{proof}
 	First let us show that there is no root on the unit circle. Suppose that there is a root $r = e^{i\phi}$ of $P$, then the equation $P(r)=0$ reads
 	
 	\begin{equation*}
 	-\frac{e^{2 i\phi} }{\delta x^2} \left( (2\varepsilon + \delta x^2 - 2\varepsilon\cos\phi)s^2(z) + 4 \delta x^2 \sin^2\phi  \right) = 0
 	\end{equation*}
 	which in turn implies that
 	\begin{equation*}
 	s^2(z) = -\frac{4 \delta x^2 \sin^2\phi}{2\varepsilon + \delta x^2 - 2\varepsilon\cos\phi} \in\mathbb{R}^-,
 	\end{equation*}
\noindent
and thus $\Re(s(z)) = 0$ which is in contradiction with $\vert z \vert > 1$. 
 	
 There remains to locate the four roots with respect to unit circle. We order the roots as follows $|r_i(z)|\geq |r_{i+1}(z)|$ with $i=1,2,3$.	First, note that the constant term of $P$ is equal to $1$ which means that $\vert r_1(z) r_2(z) r_3(z) r_4(z) \vert = 1$. Necessarily, one has $|r_1(z)|$ and $|r_4(z)|<1$.  Let $s(z) \to \infty$: one has $r_1(z)  \sim - 4 \varepsilon s^2(z)$. The remaining roots $r_2,r_3,r_4$ are bounded and 	as $s(z)\to+\infty$ converge to the roots of polynomial which is defined as
 	\begin{equation*}
 	4\varepsilon r^3 - \left(4\delta x^2 + 8 \varepsilon\right) r^2 + r = 0, 
 	\end{equation*}
 So $r_4(z) \to 0$ and one can calculate directly $\dsp r_4(z) \sim -\frac{1}{4\varepsilon s^2(z)}$ as $s(z)\to+\infty$. The roots $r_2(z)$, $r_3(z)$ converge to the solution of
 	\begin{equation*}
 	\dsp r^2 - \left(2 + \frac{\delta x^2}{\varepsilon}\right) r + 1 = 0.
 	\end{equation*}
Since the discriminant $\dsp\Delta =  \left(1 + \frac{\delta x^2}{2 \varepsilon}\right)^2 - 1 > 0$, the roots are distinct and  one finds $\vert r_2(z)\vert > 1 > \vert r_3(z)\vert$ . This concludes the proof of the separation property.
 \end{proof}
 \begin{remark}\label{KN-rootscoinc}
 	The characteristic equation $P(r)=0$ can be written in the following form
 	\begin{equation*}
 	(r-1)^2 \left((r+1)^2 + 4 s^2(z) \varepsilon r \right) = 4 s^2(z) r^2 \delta x^2,
 	\end{equation*}
 	which can be rewritten as
 	\begin{equation*}
 	\dsp (r-1)^2  = \frac{4 s^2(z) r^2 \delta x^2}{(r+1)^2 + 4 s^2(z) \varepsilon r },
 	\end{equation*}
 	and so by applying the implicit function theorem, we can compute an expansion of  the roots $r_{2,3}$ bifurcating form $1$ at $\delta x=0$:
 	\begin{equation*}
		 r_{2}(z) = 1+\frac{s(z) \dx}{\sqrt{1 + \varepsilon s^2(z)}} + O(\dx^2),\quad r_{3}(z) = 1-\frac{s(z) \dx}{\sqrt{1 + \varepsilon s^2(z)}} + O(\dx^2).
 	\end{equation*}
A similar argument yields also an asymptotic expansion of $r_{1,4}$: 	
 	\begin{equation*}
	\begin{array}{ll}
	 	r_{1} = -\left(1 + 2 \varepsilon s^2(z)\right)- 2 \sqrt{\varepsilon s^2(z)(1 + \varepsilon s^2(z))} + O(\dx^2),\\[2mm]
		r_{4} = -\left(1 + 2 \varepsilon s^2(z)\right)+ 2 \sqrt{\varepsilon s^2(z)(1 + \varepsilon s^2(z))} + O(\dx^2).
		\end{array}
 	\end{equation*}
 	
 \end{remark}
 
Thanks to roots separation, we have a decomposition of solutions space into a stable subspace $E^s(z)={\rm span}(V_3; V_4)$ of solutions decreasing to $0$ as $j\to\infty$ and an unstable subspace $E^u(z)={\rm span}(V_1, V_2)$ of solutions decreasing to $0$ as $j\to-\infty$. According to  remark \ref{KN-rootscoinc}, we should pay close attention to the choice of the spatial step $\delta x$ in order to separate the roots $E^s(z)$ and $E^u(z)$. In order to 
obtain bounded solution, one must impose 
$$
\displaystyle
 \left(\widehat{w}_{J+1}, \widehat{\eta}_{J+1},  \widehat{t}_{J+1},  \widehat{v}_{J+1}\right)^T\in E^s(z),\quad  \left(\widehat{w}_1, \widehat{\eta}_1,  \widehat{t}_1,  \widehat{v}_1\right)^T\in E^u(z).
$$
which is equivalent to 
$$
\displaystyle
 \left(\widehat{w}_{J+1}, \widehat{\eta}_{J+1},  \widehat{w}_{J},  \widehat{\eta}_{J}\right)^T\in E^s(z),\quad  \left(\widehat{w}_1, \widehat{\eta}_1,  \widehat{w}_0,  \widehat{\eta}_0\right)^T\in E^u(z).
$$

%
%
 \noindent
 Let us start with the left boundary condition: the vector  $\left(\widehat{w}_0, \widehat{\eta}_0,  \widehat{w}_1,  \widehat{\eta}_1	\right)^T$ is given by
 \begin{equation}\label{KN-systleft}
 \left(
 \begin{array}{c}
 \widehat{w}_0\\
 \widehat{\eta}_0\\
 \widehat{w}_1\\
 \widehat{\eta}_1
 \end{array}
 \right) = 
 \left(
 \begin{matrix}
 1 & 1 & 1 & 1\\
 \frac{1-r_1^2}{2 \delta x r_1 s(z)} & \frac{1-r_2^2}{2 \delta x r_2 s(z)} & \frac{1-r_3^2}{2 \delta x r_3 s(z)} & \frac{1-r_4^2}{2 \delta x r_4 s(z)}\\
 r_1  & r_2 & r_3 & r_4 \\
 \frac{1-r_1^2}{2 \delta x s(z)} & \frac{1-r_2^2}{2 \delta x s(z)} & \frac{1-r_3^2}{2 \delta x s(z)} & \frac{1-r_4^2}{2 \delta x s(z)} 
 \end{matrix}
 \right) \left(
 \begin{array}{c}
 \alpha_1^\ell \\
 \alpha_2^\ell \\
 \alpha_3^\ell \\
 \alpha_4^\ell
 \end{array}
 \right).
 \end{equation}
 with $\alpha_3^\ell=\alpha_4^\ell=0$. Then from \eqref{KN-systleft} we have:
 
 \begin{equation*}
 \widehat{\eta}_0 =  \frac{1-r_1^2}{2 \delta x r_1 s(z)} \alpha_1^\ell +  \frac{1-r_2^2}{ 2 \delta x r_2 s(z)} \alpha_2^\ell, \quad  \widehat{\eta}_1 =  \frac{1-r_1^2}{2 \delta x s(z)} \alpha_1^\ell +  \frac{1-r_2^2}{2 \delta x s(z)} \alpha_2^\ell.
 \end{equation*}
 In order to determine $\alpha_1^\ell$, $\alpha_2^\ell$, we use the remaining two equations of \eqref{KN-systleft}. We set $r_1 + r_2 = S^u$ and $r_1 r_2 = P^u$: the left boundary conditions are given by 
 \begin{equation}\label{KN-leftBC}
 \begin{array}{r}
 (1 + P^u) \widehat{w}_1 =S^u \widehat{w}_0 - 2 \delta x P^u s(z) \widehat{\eta}_0,\\[2mm]
 2 \delta x\, s(z) \widehat{\eta}_1 + S^u \widehat{w}_1 = (1 + P^u) \widehat{w}_0.
 \end{array}
 \end{equation}
 
 The derivation of right boundary conditions are carried out with the same method. We set $r_3 + r _4 = S^s$ and $r_3 r_4 = P^s$. The right boundary conditions are given by:
 \begin{equation}\label{KN-rightBC}
 \begin{array}{r}
 (1 + P^s) \widehat{w}_{J+1} = S^s \widehat{w}_J - 2 \delta x\, P^s s(z) \widehat{\eta}_J,\\[2mm]
 2 \delta x\, s(z) \widehat{\eta}_{J+1} + S^s \widehat{w}_{J+1} = (1 + P^s) \widehat{w}_J.
 \end{array}
 \end{equation}
 
 The coefficients of the boundary conditions \eqref{KN-leftBC},\eqref{KN-rightBC} contain a singularity at $z=-1$, which in turn implies that the expansion coefficients for $S^s$, $P^s$, $S^u$, $P^u$ decrease slowly. In order to remove this singularity, we multiply the boundary conditions  \eqref{KN-leftBC},\eqref{KN-rightBC} by $(1+z^{-1})^q$, where the power $q$ depends on the order of a pole  $z=-1$ of the coefficients. For example, as we have seen the unstable root $r_1$ has the following asymptotic behaviour $r_1 \sim s^2(z)$ (see Proposition \ref{KN-PrSepar}). For stabilization, the coefficient $- 2\delta x P^u s(z)$ needs to be multiplied by $(1+z^{-1})^3$. The roots $r_3$, $r_4$ stay bounded as well as $P^s$ and $S^s$ and, therefore, we need to deal only with the singularity of $s(z)$, and $q=1$. We set $z^{-1} = x$ and obtain the following boundary conditions with coefficients decreasing faster which ensures stability with respect to round off errors: 
 \begin{equation*}\label{KN-BCmult}
 \begin{array}{c}
  \dsp((1+x)^3 + (1+x)\tilde{P}^u) \widehat{w}_1 =(1+x) \tilde{S}^u \widehat{w}_0 - \frac{4\delta x}{\delta t} (1-x) \tilde{P}^u \widehat{\eta}_0,\\[2mm]
  \dsp\frac{4\delta x}{\delta t} (1-x^2)  \widehat{\eta}_1 + \tilde{S}^u \widehat{w}_1 = ((1+x)^2  + \tilde{P}^u) \widehat{w}_0.\\[2mm]
  \dsp((1+x) + (1+x) P^s) \widehat{w}_{J+1} = (1+x)S^s \widehat{w}_J - \frac{4\delta x}{\delta t} (1-x) P^s \widehat{\eta}_J,\\[2mm]
  \dsp\frac{4\delta x}{\delta t} (1-x)  \widehat{\eta}_{J+1} + (1+x)S^s \widehat{w}_{J+1} = ((1+x)  + (1+x)P^s) \widehat{w}_J.
 \end{array}
 \end{equation*}
 
In order to invert $\mathcal{Z}-$transform, it is required to find the coefficients in the expansions of $S^s$, $P^s$, $\tilde{S}^u$, $\tilde{P}^u$ which are defined as
 \begin{equation*}
 \begin{array}{c}
 \dsp S^s(x)  = \sum\limits_{n>0} s^s_n x^n, \quad P^s(x)  = (1+x) \sum\limits_{n>0} p^s_n x^n,\\[5mm]
 \dsp\tilde{S}^u(x)  = (1+x)^2 S^u(x) = (1+x)^2 \sum\limits_{n>0} s^u_n x^n = \sum\limits_{n>0} \tilde{s}^u_n x^n, \\[5mm]
 \dsp\tilde{P}^u(x)  = (1+x)^2 P^u(x) =  (1+x)^2 \sum\limits_{n>0} p^u_n x^n = \sum\limits_{n>0} \tilde{p}^u_n x^n.
 \end{array}
 \end{equation*}
 \noindent
We follow the procedure proposed in \cite{BNS} and use the relation between the roots and coefficients of $P$. More precisely we have
 \begin{equation*}
 \begin{array}{r}
 S^s + S^u = -4\varepsilon s^2(x),\\[2mm]
 P^u + S^u S^s + P^s = - (2 + 4  s^2(x) (\delta x^2 + 2 \varepsilon)),\\[2mm]
 P^uS^s + P^s S^u = -4\varepsilon s^2(x),\\[2mm]
 P^u P^s = 1.
 \end{array}
 \end{equation*}
Then, the system satisfied by $S^s$, $P^s$, $\tilde{S}^u$, $\tilde{P}^u$ is given by
 \begin{equation*}
 \begin{array}{r}
 (1+x)^2 S^s + \tilde{S}^u = -16 \varepsilon(1-x)^2/\delta t^2,\\[2mm]
 \tilde{P}^u + \tilde{S}^u S^s + (1+x)^2 P^s = - (2(1+x)^2 + 16 (1-x)^2 (\delta x^2+ 2 \varepsilon)/\delta t ^2),\\[2mm]
 \tilde{P}^u S^s + P^s \tilde{S}^u =-16 \varepsilon (1-x)^2/\delta t^2,\\[2mm]
 \tilde{P}^u P^s = (1+x)^2.
 \end{array}
 \end{equation*}
By substituting the expansion of $S^s$, $P^s$, $\tilde{S}^u$, $\tilde{P}^u$ in this system, one finds
 \begin{equation}\label{KN-coefsyst}
 \begin{array}{r}
 \dsp s^s_n + \tilde{s}^u_n = -(2 s^s_{n-1} +  s^s_{n-2}) - 16 \varepsilon  \sigma^n/ \delta t^2,\\[2mm]
 \dsp\tilde{p}^u_n + s^s_0 \tilde{s}^u_n + \tilde{s}^u_0 s^s_n + p^s_n = -(2 p^s_{n-1} + p^s_{n-2}) - \sum\limits_{k=1}^{n-1} s^s_k \tilde{s}^u_{n-k} - \kappa^n,\\[2mm]
 \dsp s^s_0 \tilde{p}^u_n + \tilde{p}^u_0 s^s_n + p^s_0 \tilde{s}^u_n +  \tilde{s}^u_0 p^n_s = - \sum\limits_{k=1}^{n-1} s^s_k \tilde{p}^u_{n-k} - \sum\limits_{k=1}^{n-1} p^s_k \tilde{s}^u_{n-k} - 16 \varepsilon  \sigma^n/ \delta t^2,\\[2mm]
 \dsp p^s_0 \tilde{p}^u_n + \tilde{p}^u_0 p^s_n = - \sum\limits_{k=1}^{n-1} p^s_k \tilde{p}^u_{n-k} + \zeta_n.
 \end{array}
 \end{equation}
 where the sequence $\sigma^n$, $\kappa^n$ are given by formulas
 \begin{equation*}
 \begin{array}{c}
 \sigma^n = \delta_0 - 2 \delta_1 + \delta_2,\\
 \zeta^n = \delta_0 + 2 \delta_1 + \delta_2,\\
 \kappa^n = (2 + 16(\delta x^2 + 2 \varepsilon)/ \delta t^2) \delta_0 - (4 - 32(\delta x^2 + 2 \varepsilon)/ \delta t^2)\delta_1 - (2 + 16(\delta x^2 + 2 \varepsilon)/ \delta t^2) \delta_2,
 \end{array}
 \end{equation*}
 and $\delta_0 = (1, 0, \dots 0, \dots)$, $\delta_1 = (0, 1, 0, \dots 0, \dots)$,  $\delta_2 = (0, 0, 1, 0 \dots 0, \dots)$. The quantities $\tilde{s}^s_0$, $\tilde{s}^u_0$, $\tilde{p}^s_0$, $\tilde{p}^u_0$ are found directly as the roots of $P$ for $z^{-1} = x = 0$, and the resolution of \eqref{KN-coefsyst} is implemented numerically. Now there just remains to invert the boundary conditions \eqref{KN-BCmult}, one finds on the left
 
 \begin{equation}\label{KN-BCleft}
 \begin{array}{c}
  \dsp (1 + \tilde{p}^u_0) w_1^{n+1}  - \tilde{s}_0^u w_0^{n+1} + \frac{4\delta x}{\delta t} \tilde{p}_0^u \eta_0^{n+1} = -(3 + \tilde{p}^u_1 +\tilde{p}^u_0) w_1^{n} + (\tilde{s}^u_1 +  \tilde{s}^u_0)w_0^{n} - \dsp-  \frac{4\delta x}{\delta t}(\tilde{p}^u_1 - \tilde{p}^u_0)\eta_0^{n} - \\[3mm] \dsp 3 w_1^{n-1} - w_1^{n-2} - \sum\limits_{k=1}^{n} (\tilde{p}^u_{k+1} +\tilde{p}^u_{k} )w_1^{n-k} + \sum\limits_{k=1}^{n} (\tilde{s}^u_{k+1} +  \tilde{s}^u_{k} )w_0^{n-k} -  \frac{4\delta x}{\delta t}\sum\limits_{k=1}^{n} (\tilde{p}^u_{k+1} - \tilde{p}^u_{k})\eta_0^{n-k} ,\\[3mm]
 \dsp \frac{4\delta x}{\delta t}  \eta_1^{n+1} + \tilde{s}_0^u w_1^{n+1} - (1 + \tilde{p}^u_0) w_0^{n+1}    = \\[3mm] \dsp -\tilde{s}^u_1  w_1^{n} + (2 + \tilde{p}^u_1 ) w_0^{n}  + \frac{4\delta x}{\delta t}\eta_1^{n-1}  + w_0^{n-1} - \sum\limits_{k=1}^{n} \tilde{s}^u_{k+1} w_1^{n-k} + \sum\limits_{k=1}^{n} \tilde{p}^u_{k+1}  w_0^{n-k},
  \end{array}
 \end{equation}
and on the right:  \begin{equation}\label{KN-BCright}
\begin{array}{c}
\dsp (1 + p^s_0) w_{J+1}^{n+1}  - s_0^s w_J^{n+1} + \frac{4\delta x}{\delta t} p_0^s \eta_J^{n+1} = - w_{J+1}^{n} + (s^s_{1} + s^s_{0})w_J^{n} - (p^s_{1} + p^s_{0})w_{J+1}^{n} -  \frac{4\delta x}{\delta t}(p^s_{1} - p^s_{0})\eta_J^{n} \\[3mm] \dsp - \sum\limits_{k=1}^{n} (p^s_{k+1} + p^s_{k})w_{J+1}^{n-k} + \sum\limits_{k=1}^{n} (s^s_{k+1} + s^s_{k})w_J^{n-k} -  \frac{4\delta x}{\delta t}\sum\limits_{k=1}^{n} (p^s_{k+1} - p^s_{k})\eta_J^{n-k} ,\\[3mm]
 \dsp \frac{4\delta x}{\delta t}  \eta_{J+1}^{n+1} + s_0^s w_{J+1}^{n+1} -  (1 + p^s_0) w_J^{n+1}  = \frac{4\delta x}{\delta t}\eta_{J+1}^{n}  - (s_0^s + s_1^s) w_{J+1}^{n} + (1 + p^s_1 + p^s_0) w_J^{n} \\[3mm] + \dsp\sum\limits_{k=1}^{n} (p^s_{k+1} + p^s_{k}) w_J^{n-k} - \sum\limits_{k=1}^{n} (s^s_{k+1} + s^s_k)  w_{J+1}^{n-k}.
 \end{array}
 \end{equation}


\subsection{Consistency theorem}

We show that the discrete boundary conditions \eqref{KN-BCleft}, \eqref{KN-BCright} are consistent of  order $O(\dt^2 +\dx^2)$.
 
  \begin{theorem}
 	Let $\eta$, $w$ be a smooth solution \eqref{KN-inisyst} and \eqref{KN-conds}.  We define the \ZT transform of $f(\cdot,x)$ for all $x\in[x_\ell, x_r]$ by
 	\begin{equation*} 
 	\forall z \neq 0, \quad \widehat{f}(z,x) = \sum\limits_{n = 0}^{\infty} \frac{f(n\dt, x)}{z^{n}}.
 	\end{equation*}
 	\noindent
 	For all compact $K\subset\mathbb{C}^+$ and for all $s \in K$:
 	\begin{equation*}
 	\begin{array}{r}
 	(1 + r_1 r_2) \widehat{w}(e^{s\dt},x_\ell+\dx) - (r_1 + r_2) \widehat{w}(e^{s\dt},x_\ell) + 2 \delta x r_1 r_2 s(e^{s\dt}) \widehat{\eta}(e^{s\dt}, x_\ell) = O(\dt^2 + \dx^2),\\[2mm]
 	2 \delta x s(e^{s\dt}) \widehat{\eta}(e^{s\dt},x_\ell+\dx) + (r_1 + r_2) \widehat{w}(e^{s\dt},x_\ell+\dx) -  (1 + r_1 r_2) \widehat{w}(e^{s\dt},x_\ell) = O(\dt^2 + \dx^2),
 	\end{array}
 	\end{equation*}
 	
 	\begin{equation*}
 	\begin{array}{r}
 	(1 + r_3 r_4) \widehat{w}(e^{s\dt},x_r) - (r_3 + r_4) \widehat{w}(e^{s\dt}, x_r-\dx) + 2 \delta x r_3 r_4 s(e^{s\dt}) \widehat{\eta}(e^{s\dt},x_r-\dx) = O(\dt^2 + \dx^2),\\[2mm]
 	2 \delta x s(e^{s\dt}) \widehat{\eta}(e^{s\dt}, x_r) + (r_3 + r_4) \widehat{w}(e^{s\dt}, x_r) - (1 + r_3 r_4) \widehat{w}(e^{s\dt},x_r-\dx)  = O(\dt^2 + \dx^2).
 	\end{array}
 	\end{equation*}
	where $r_i$, $i=1..4$ are the roots of polynomial \eqref{KN-poly} such that $|r_1|\geq |r_2|>1>|r_3|\geq |r_4|$.
 \end{theorem}
 
 
\begin{proof}
 	The proof of this theorem is similar to the proof of theorem \eqref{ConsistStag}. Though the 
explicit expressions for the roots $r_i$ are exceedingly lengthy and useless. Instead, we consider asymptotic expansions of the roots as $\dx, \dt\to 0$. Recall that $(r_i)_{i=1,\dots,4}$ expand as
 	 \begin{equation*}
 	 \begin{array}{l}
		\dsp r_{2} = 1 +\frac{s(z) \dx}{\sqrt[+]{1 + \varepsilon s^2(z)}} + O(\dx^2),\quad r_{3} = 1 - \frac{s(z) \dx}{\sqrt[+]{1 + \varepsilon s^2(z)}} + O(\dx^2),\\[5mm]
		r_{1} = -1 - 2 \varepsilon s(z)^2- 2 \sqrt[+]{\varepsilon s^2(z)(1 + \varepsilon s^2(z))} + O(\dx^2),\\[2mm]       r_{4} = -1 - 2 \varepsilon s(z)^2+ 2 \sqrt[+]{\varepsilon s^2(z)(1 + \varepsilon s^2(z))} + O(\dx^2).
		 \end{array}
	\end{equation*}
Let us denote $e_1(\dt,\dx)$ the consistency error associated to the first boundary condition:
$$
\dsp
e_1(\dt,\dx)=(1 + r_1 r_2) \widehat{w}(e^{s\dt},x_\ell+\dx) - (r_1 + r_2) \widehat{w}(e^{s\dt},x_\ell) + 2 \delta x r_1 r_2 s(e^{s\dt}) \widehat{\eta}(e^{s\dt}, x_\ell)
$$
and introduce $R_1, R_2$ such that
$$
\displaystyle
r_1=R_1+O(\dx^2),\qquad r_2=1+R_2\,\dx+O(\dx^2).
$$
\noindent
The consistency error $e_1$ reads:
{\setlength\arraycolsep{1pt}
\begin{eqnarray}
\dsp
e_1&=&(1+R_1+R_1R_2\dx)\widehat{w}(e^{s\dt},x_\ell+\dx)-(1+R_1+R_2\dx)\widehat{w}(e^{s\dt},x_\ell)\nonumber\\
\dsp
&&+2\dx s(e^{s\dt})R_1\widehat{\eta}(e^{s\dt},x_\ell)+O(\dx^2)\nonumber\\[2mm]
\dsp
&=&\dx\left((1+R_1)\partial_x\widehat{w}(e^{s\dt},x_\ell)+R_2(R_1-1)\widehat{w}(e^{s\dt},x_\ell)+2s(e^{s\dt})\widehat{\eta}(e^{s\dt},x_\ell)\right)+O(\dx^2)\nonumber\\[2mm]
\dsp
&=&\dx(R_1+1)\left(\partial_x\widehat{w}+s(e^{s\dt})\widehat{\eta}\right)(e^{s\dt},x_\ell)+\dx(R_1-1)\left(R_2\widehat{w}+s(e^{s\dt})\widehat{\eta}\right)(e^{s\dt},x_\ell)+O(\dx^2).\nonumber
\end{eqnarray}
}

\noindent
Recall that the function $s(z)$ with $z = e^{s\dt}$ is approximated by $s(e^{s\dt}) = s + O(\dt^2)$. By applying the relation \eqref{LApprox}  between \ZT transform and Laplace transform, we find that
$$
\dsp
e_1=\frac{\dx}{\dt}(R_1+1)\left(\partial_x\LL(w)+s \LL(\eta)\right)(s,x_\ell)+\frac{\dx}{\dt}(R_1-1)\left(R_2\LL(w)+s\LL(\eta)\right)(s, x_\ell)+O(\dx^2+\dt^2).
$$
\noindent
Since the smooth solution $\eta, w$ is a solution of \eqref{KN-inisyst}, one has $\partial_x\LL(w)+s \LL(\eta)=0$. Moreover $\eta, w$ satisfies \eqref{KN-conds} so that $\left(R_2\LL(w)+s\LL(\eta)\right)(s, x_\ell)=0$. As a result, one has $e_1=O(\dt^2+\dx^2)$. We proceed the same way for the other consistency errors. This conludes the proof of the proposition.
\end{proof}

We observed numerically that the coefficients involved in \eqref{KN-BCright} and \eqref{KN-BCleft} decrease as $n^{-3/2}$. The coefficients are plotted on the figure \ref{coef}, $b$. One finds similar decay properties for the linear Korteweg-de Vries equation \cite{BNS}, Benjamin-Bona-Mahony equation \cite{BMN} or the Schr\"odinger equation \cite{ET}.  The discrete boundary conditions \eqref{KN-BCright} and \eqref{KN-BCleft} are thus stable with respect to round off errors. 

 In the next section we will discuss the results of numerical simulations for equation \eqref{KN-Disc1eq} with the boundary conditions \eqref{KN-left}, \eqref{KN-right}  and for the system \eqref{KN-CrNicDis} with the  boundary conditions \eqref{KN-BCleft}, \eqref{KN-BCright}.
  
\section{Numerical results} \label{Num}

In this section we present a numerical validation of the discretized transparent boundary conditions
through various tests. First we validate the boundary conditions for a Gaussian initial data. Different dispersion properties are analysed for a wave packet as initial datum. This analysis is based on the dispersion relation corresponding to the linearised Green-Naghdi equation. All test are carried out for  both types of boundary conditions on a staggered and on a collocated grid. Finally, we show how to inject a (planar) wave into the computational domain. To validate the efficiency of the artificial boundary conditions we perform a numerical analysis of the  approximation error. The tests show second order of approximation with respect to time and space. Let us introduce first the numerical implementation of the  numerical methods considered in this paper.
 
\subsection{Numerical implementation}
\subsubsection{Staggered grid}
We present a numerical strategy to solve the problem on a staggered grid. The discretization \eqref{KN-CrNicDis} is equivalent to scheme \eqref{KN-Disc1eq} and the conditions \eqref{KN-left}, \eqref{KN-right} are written for the values of velocity $w_{0,J+1}^n$. It remains to reconstruct the values for free surface elevations $\eta_{j+1/2}^{n+1}$, $j\in(0,J)$. 
By taking into account the boundary condition and setting 
 \begin{equation*}
 \begin{array}{c}
 \Lambda_- = \Lambda + 2\dx^2 - 2\dx \sqrt{\Gamma}, \quad \Lambda_+ = \Lambda + 2\dx^2 + 2\dx \sqrt{\Gamma}\\[2mm]
 \mu_- = \mu + 2\dx^2 - \dx \sqrt{\Gamma}(v+1), \mu_+ = \mu + 2\dx^2 + \dx \sqrt{\Gamma}(v+1),
 \end{array}
 \end{equation*} 
 the full numerical step written as a one time step method reads
 \begin{equation*}
 M_{n+1} W^{n+1} = 2 M_{n} W^{n} - M_{n+1} W^{n-1} + V, n\in\N,
 \end{equation*}
 where $W^{n+1} = [w_0^{n+1},\, \dots\, , w_{J+1}^{n+1}]^\top$ is the unknown vector, and the matrices $M_n, M_{n+1}$ are defined as:
 \begin{equation*}
 M_{n+1} =  \begin{bmatrix}
 -\Lambda_+ &    \Lambda & &  &\\
 -a_+         & 1 + 2a_+ & -a_+      &        \\
 & \ddots  & \ddots    & \ddots   & 	   \\
 &             & -a_+       & 1 + 2a_+ & -a_+\\
 & &    &-\Lambda_- & \Lambda 
 \end{bmatrix}, 
 \quad 
 M_{n} =  \begin{bmatrix}
 -\mu_+ &    \mu & &  &\\
 -a_+         & 1 + 2a_+ & -a_+      &        \\
 & \ddots  & \ddots    & \ddots   & 	   \\
 &             & -a_+       & 1 + 2a_+ & -a_+\\
 &   &     &-\mu_- & \mu
 \end{bmatrix}.
 \end{equation*}
 The vector $V_n$ on the right hand side has only two non-zero components:
 \begin{equation*}
 V _n=   \begin{bmatrix}
 \dsp 2 \dx\sqrt{\Gamma} \left( (\PP_2 - 2 v^2 + v)  w_0^{n-1} + \sum_{k =2}^{n} s_k(v)  w_0^{n-k} \right) \\
 \vdots\\
 \dsp -2 \dx\sqrt{\Gamma} \left( (\PP_2 - 2 v^2 + v)  w_J^{n-1} + \sum_{k =2}^{n} s_k(v)  w_J^{n-k} \right)\\
 \end{bmatrix}.
 \end{equation*}
 The matrix $M_{n+1}$ is easily proved to be invertible (for $\dx$ small enough) and the solution vector at time $t_{n+1}$ is given by
 \begin{equation*}
 W^{n+1} = M_{n+1}^{-1}  (2 M_{n} W^{n} + V) - W^{n-1}, n\in\N,
 \end{equation*}
 so that the velocity components can be computed at each time step.\\

Once the velocity field is computed, there remains to reconstruct the values of free-surface elevation $\eta$. This can be done by solving the first equation of \eqref{KN-CrNicDis}. Since the velocity at iterations $n-$ and $(n+1)$ is known, one finds
 \begin{equation*}
 \dsp \eta_{j+1/2}^{n+1} = \eta_{j+1/2}^{n} - \frac{\dt}{2} \left( \frac{w_{j+1}^{n+1} - w_{j}^{n+1}}{ \delta x} +  \frac{w_{j+1}^{n} - w_{j}^{n}}{ \delta x} \right).
 \end{equation*}
 Note that this equation has no influence on the velocity calculations and should be solved simply for the correct description of the water wave problem.
 
 We need to set the initial values for velocity $W^{0}$ at $t=0$ and $W^{1}$ at $t=\dt$. In order to take into account the physics of the problem the initial conditions should be imposed for velocity $W^0$ and elevation $\eta^0$. To find a value for $W^{1}$ at $t=\dt$ we use the Taylor expansion in the vicinity of $t=0$:
 \begin{equation*}
 (w - \varepsilon w_{xx})\mid_{t = \dt} = (w - \varepsilon w_{xx})\bigr|_{t = 0} + \dt (w - \varepsilon w_{xx})_t \bigr|_{t = 0} + \frac{\dt^2}{2} (w - \varepsilon w_{xx})_{tt}\bigr|_{t = 0} + O(\dt^2),
 \end{equation*}
 using the continuous equations \eqref{KN-inisyst} one finds
 \begin{equation}\label{KN-ini}
 (w - \varepsilon w_{xx})\bigr|_{t = \dt} = \left(w - \left(\varepsilon -\frac{\dt^2}{2}\right) w_{xx}  \right)\bigr|_{t = 0}  + \dt\left(\eta_x)\right)\bigr|_{t = 0}.
 \end{equation}
 The discretization of \eqref{KN-ini} gives the linear system for the requested value. Note that the order of approximation for values $W(t = \dt)$ is the same than the numerical scheme itself. Though we
 have to take care of the choice of the time step with respect to values of $\varepsilon$ especially if $\varepsilon$ is  small  ($10^{-4}$, $10^{-5}$).

\subsubsection{Collocated grid}

 We rewrite in a matrix form the discrete equations \eqref{KN-CrNicDisCol} on a collocated grid coupled with the boundary conditions derived in section \ref{Disc}. We have:
 \begin{equation*}
  \begin{bmatrix}
  A_{n+1} &    B_{n+1} \\
 C_{n+1} & D_{n+1}
 \end{bmatrix} \left( \begin{array}{c}
 \eta\\
 w
 \end{array}\right)^{n+1} = \begin{bmatrix}
 A_{n} &    B_{n} \\
 C_{n} & D_{n}
 \end{bmatrix} \left( \begin{array}{c}
 \eta\\
 w
 \end{array}\right)^{n} + \vec{V}_n
 \end{equation*}
 where the matrices $M_{n+1}$, $M_n$ size of $2\times(J+2) \times 2\times(J+2)$ are block matrices. The blocks for index $(n+1)$ are defined as follows
\begin{equation*}
 	A_{n+1}  =  \begin{bmatrix}
 	\tilde{p}_0^u/c &  	 &       &  &\\
 	    0         & 1          &   0  &  &\\
 	& \ddots  & \ddots   & \ddots   & 	   \\
 	&             &  0 & 1 & 0\\
 	0&    	1/c   & &0 & 0
 	\end{bmatrix},  \quad
 	B_{n+1}  =  \begin{bmatrix}
 	-\tilde{s}_0^u &   1 + \tilde{p}_0^u &       &  &\\
 	-c         & 0         &   c   &  &\\
 	& \ddots  & \ddots   & \ddots   & 	   \\
 	&             &  -c        & 0         &  c\\
 	-(1+\tilde{p}_0^u)&   \tilde{s}_0^u &    & & 0
 	\end{bmatrix},  
\end{equation*}
\begin{equation*}
C_{n+1}  =  \begin{bmatrix}
0 &   0 &       &  p_0^s/c & 0\\
-c         & 0         &  c   &  &\\
& \ddots  & \ddots   & \ddots   & 	   \\
&             &  -c         & 0         &  c\\
0&    &    & 0 & 1/c
\end{bmatrix},   \quad
D_{n+1}  =  \begin{bmatrix}
0 &     & &  -s_0^s & 1+p_0^s\\
-a        & 1 + 2a & -a      &        \\
& \ddots  & \ddots    & \ddots   & 	   \\
&             & -a      & 1 + 2a & -a\\
&   &     & - (1+p_0^s) & s_0^s
\end{bmatrix},  
\end{equation*}
and for $n$:
\begin{equation*}
A_{n}  =  \begin{bmatrix}
-(\tilde{p}_1^u - \tilde{p}_0^u)/c &  	 &       &  &\\
0         & 1          &   0  &  &\\
& \ddots  & \ddots   & \ddots   & 	   \\
&             &  0 & 1 & 0\\
0&    	0  & &0 & 0
\end{bmatrix},   \quad
B_{n}  =  \begin{bmatrix}
\tilde{s}_0^u+	\tilde{s}_1^u &   -(3 + \tilde{p}_0^u+	\tilde{p}_1^u) &       &  &\\
c         & 0         &   -c   &  &\\
& \ddots  & \ddots   & \ddots   & 	   \\
&             &  c        & 0         &  -c\\
2+\tilde{p}_1^u&   -\tilde{s}_1^u &    & & 0
\end{bmatrix},  
\end{equation*}
\begin{equation*}
C_{n}  =  \begin{bmatrix}
0 &   0 &       &  -(p_1^s - p_0^s)/c & 0\\
c         & 0         &  -c   &  &\\
& \ddots  & \ddots   & \ddots   & 	   \\
&             &  c         & 0         &  -c\\
0&    &    & 0 & 1/c
\end{bmatrix},   \quad
D_{n}  =  \begin{bmatrix}
0 &     & &  (s_0^s + s_1^s) & -(1+p_0^s+p_1^s)\\
-a        & 1 + 2a & -a      &        \\
& \ddots  & \ddots    & \ddots   & 	   \\
&             & -a      & 1 + 2a & -a\\
&   &     & 1+p_0^s+p_1^s & -(s_0^s + s_1^s)
\end{bmatrix}.
\end{equation*}
We have denoted  $c = 4\dx/\dt$ and  $a=\varepsilon/\dx^2$. It follows from the form of the boundary conditions that the vector $\vec{V}_n$ on the right hand side contains the previous time-iteration values of the functions $\eta_j^n$, $w_j^n$:
 \begin{equation*} 
 \dsp V(0) = -3w_1^{n-1} - w_1^{n-2} - \sum\limits_{k = 1}^{n} (\tilde{p}_{k+1}^u + \tilde{p}_{k}^u)) w_1^{n-k} + \sum\limits_{k = 1}^{n} (\tilde{s}_{k+1}^u + \tilde{s}_{k}^u)) w_0^{n-k}  - \frac{4\dx}{\dt} \sum\limits_{k = 1}^{n} (\tilde{p}_{k+1}^u - \tilde{p}_{k}^u) \eta_0^{n-k},
 \end{equation*}
 \begin{equation*} 
 V(J+1) =  \frac{4\dx}{\dt} \eta_1^{n-1} -  \sum\limits_{k = 1}^{n} \tilde{s}_{k+1}^u w_1^{n-k}  + w_0^{n-1} + \sum\limits_{k = 1}^{n} \tilde{p}_{k+1}^u w_0^{n-k} 
 \end{equation*}
 \begin{equation*} 
 V(J+2) = \frac{4\dx}{\dt}  \sum\limits_{k = 1}^{n}(p^s_{k+1} - p^s_k) \eta_J^{n-k} -  \sum\limits_{k = 1}^{n}(p^s_{k+1} + p^s_k) w_{J+1}^{n-k} + \sum\limits_{k = 1}^{n}(s^s_{k+1} + s^s_k) w_{J}^{n-k} ,
 \end{equation*}
 \begin{equation*} 
 V(2(J+2)) = -\sum\limits_{k = 1}^{n}(s^s_{k+1} + s^s_k) w_{J+1}^{n-k}  + \sum\limits_{k = 1}^{n}(p^s_{k+1} + p^s_k) w_{J}^{n-k} ,
 \end{equation*}
 \begin{equation*} 
 V(j) = 0, \quad j = 1\,..\,J,(J+3)\,..\,2(J+2).
 \end{equation*}
 
  \begin{figure}[t]
 	\begin{center}
 		\begin{minipage}[h]{0.49\linewidth}
 			\center{ \includegraphics[width=1\linewidth]{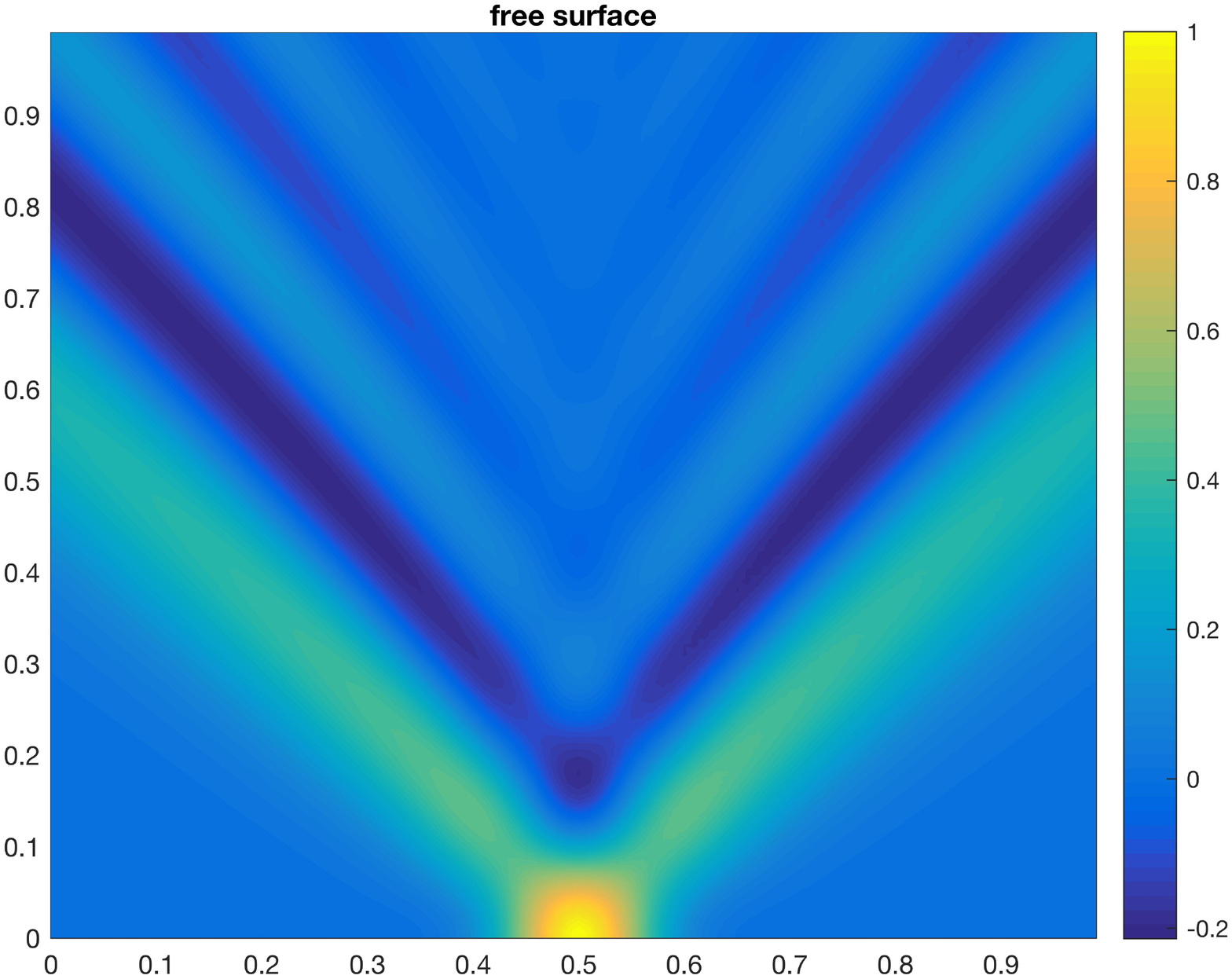}}
 		\end{minipage}
 		\hfill
 		\begin{minipage}[h]{0.49\linewidth}
 			\center{\includegraphics[width=1\linewidth]{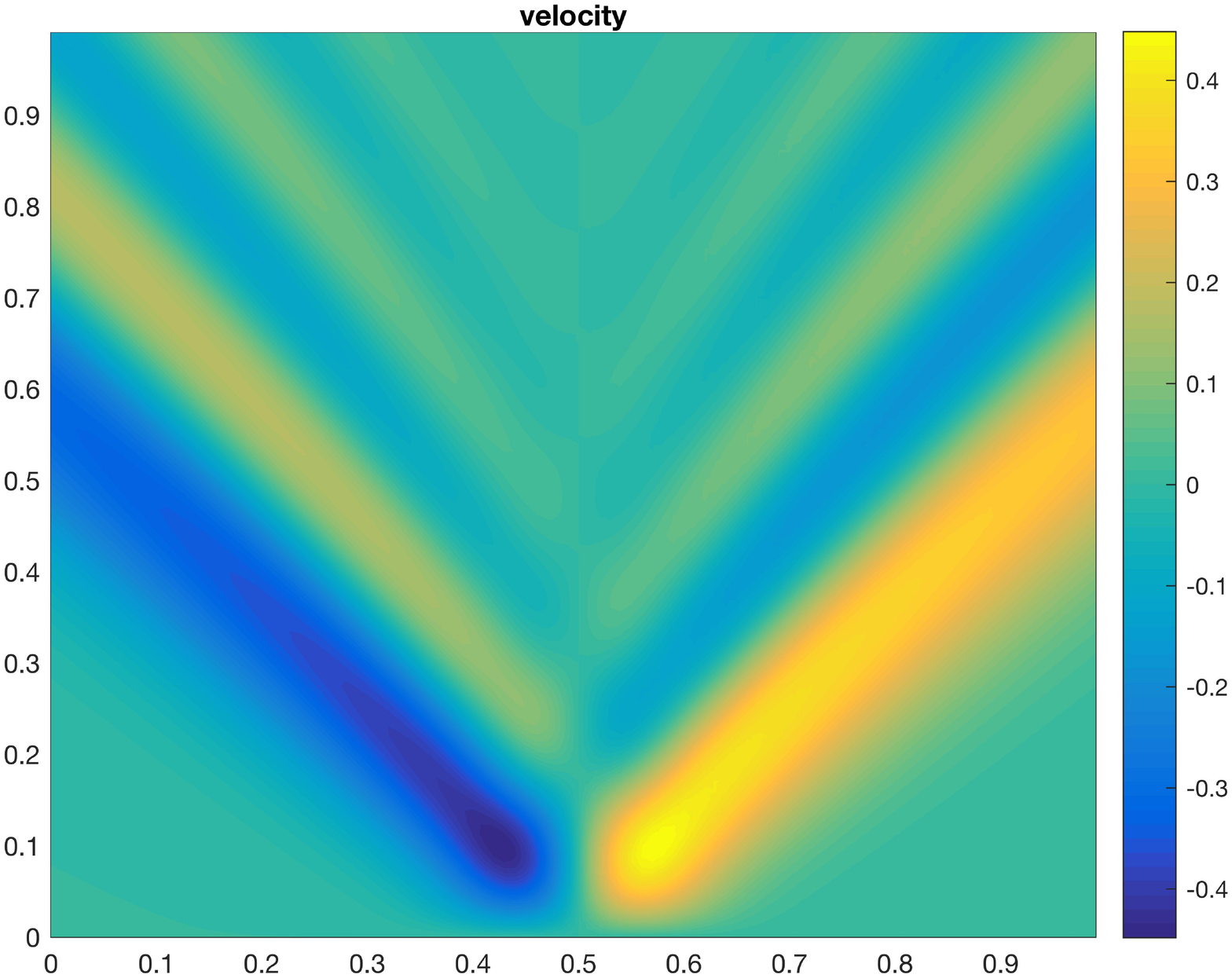}}
 		\end{minipage}
 		\vfill
 		\begin{minipage}[h]{0.49\linewidth}
 			\center{ \includegraphics[width=1\linewidth]{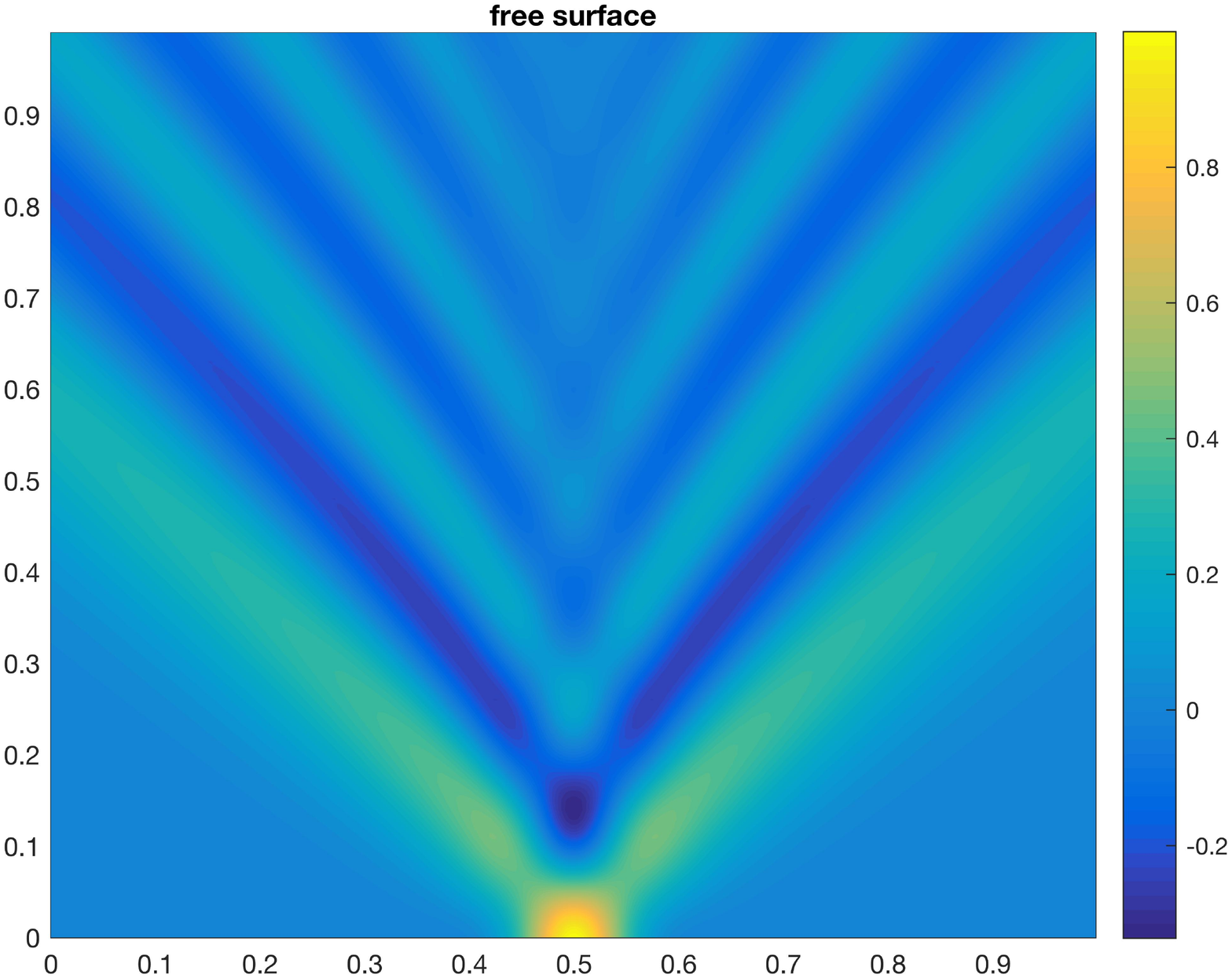} \\ {\it a}}
 		\end{minipage}
 		\hfill
 		\begin{minipage}[h]{0.49\linewidth}
 			\center{\includegraphics[width=1\linewidth]{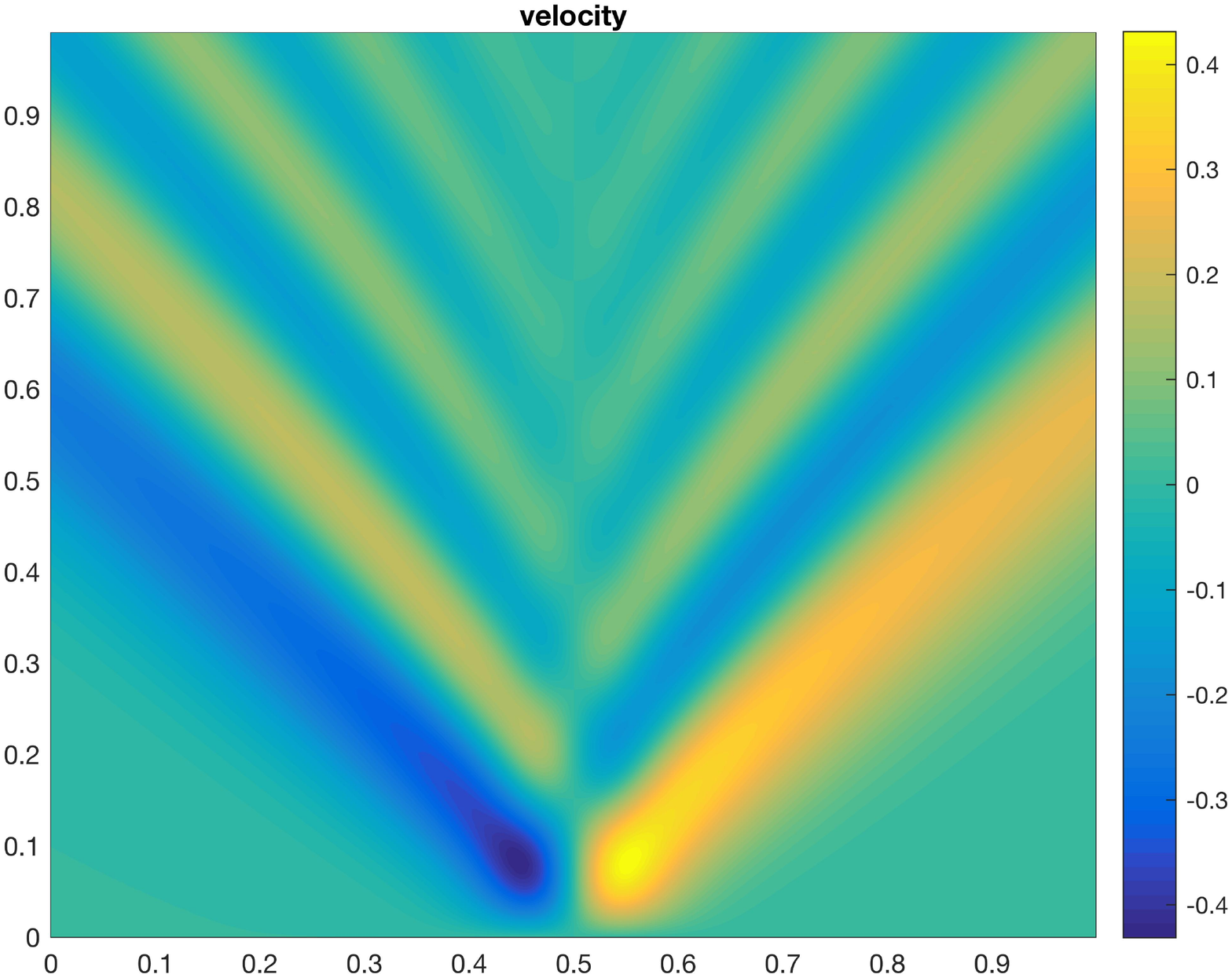} \\ {\it b}}\\
 		\end{minipage}
 	\end{center}
 	\caption{Numerical solution on a staggered(up) and Collocated (down) grids:  Evolution of (a) the surface elevation, (b) the fluid velocity for $\delta x = 10^{-3}$, $\delta t = 10^{-2}$, $\varepsilon = 10^{-3}$}
 	\label{gaus}
 \end{figure}
 
 \subsection{Gaussian initial distribution}
 
In this section, we show numerical results when we take a Gaussian initial distribution for the free surface elevation and zero distribution for velocity
 \begin{equation*}
 	\eta_0(x) = \exp(-400\times (x - 1/2)^2),  \quad w_0(x) = 0,
 \end{equation*}
whereas the computational domain $(t,x) \in [0,1] \times [0,1]$ is meshed with $N \times (J+2)$ nodes. We first show that there is no  reflection on the boundaries of the computational domain. We present results both for staggered and collocated spatial grids. The velocity and free surface evolution are  shown on the $(x,t)$-plane on the Figure, \ref{gaus}. Following the numerical strategy described at the beginning of this section we have reconstructed the value for $w(t = \dt)$ from initial datum for the method on a staggered grid.

Let us comment the results found. Recall that the dispersion relation associated to the \eqref{KN-inisyst} is written as
 \begin{equation}\label{disprel}
	 \omega^2(k) = \frac{k^2}{1 + \varepsilon k^2},
 \end{equation}
 there are two solutions for $\omega(k)$, that corresponds to the fact that Green-Naghdi system describes bi-directional propagation of waves, just as we can see on the Figure, \ref{gaus}. On the left Figure, \ref{PhaseGroup} the positive solution of dispersive relation is plotted, there is more diversity for the values of $\omega(k)$ and value for the same $k$ is more important as $\varepsilon > 0$ increases. Other properties are related to the difference between the group and phase velocities. From dispersive relation \eqref{disprel} we conclude,
  \begin{equation*}
 v_\varphi (k) = \frac{\omega(k)}{k} = \frac{1}{\sqrt{1 + \varepsilon k^2}}, \quad v_g (k) = \frac{d\omega(k)}{dk} = \frac{1}{(1 + \varepsilon k^2)^{3/2}}.
 \end{equation*}
 Group velocity is always less than phase velocity (see right Figure \ref{PhaseGroup}). 
 
 \begin{figure}[t]
 	\begin{center}
 		\begin{minipage}[h]{0.49\linewidth}
 			\center{ \includegraphics[width=1\linewidth]{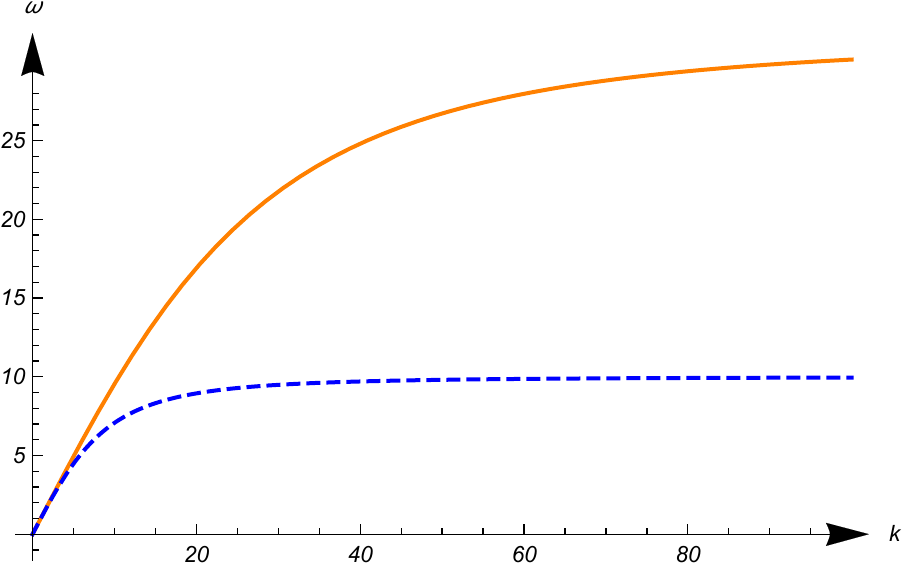}}
 		\end{minipage}
 		\hfill
 		\begin{minipage}[h]{0.49\linewidth}
 			\center{\includegraphics[width=1\linewidth]{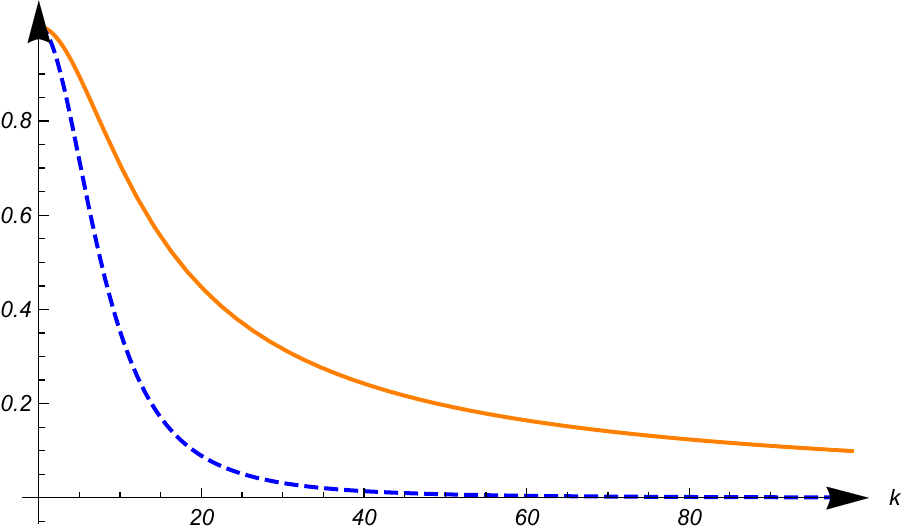}}
 		\end{minipage}
 	\end{center}
 	\caption{Positive solution of dispersive relation $\omega(k)$ for $\varepsilon = 10^{-3}$ (continued), $\varepsilon = 10^{-3}$ (dashed) (left) and phase (continued) and group (dashed) velocities for $\varepsilon = 10^{-3}$ (right).}
 	\label{PhaseGroup}
 \end{figure}
 
 \subsection{Wave packet}

In order to observe more clearly the dispersive behavior of the Green-Naghdi system and demonstrated the applicability of the constructed  boundary conditions for other tests, we consider the solution of \eqref{KN-inisyst} with the next initial datum
\begin{equation}\label{iniWP}
	\eta_0(x) = \exp(-400\times(x - 1/2)^2) \sin(20\pi x),  \quad w_0(x) = 0.
\end{equation}
For the different value of $\varepsilon$,  the dispersive properties are not the same. Results are presented on the figure, \ref{WP}. As dispersive effects are more important for $\varepsilon = 10^{-3}$ we have more diversity for frequency values, but for smaller value $\varepsilon = 10^{-4}$ the behaviour of the solution is closer to the  solutions of the hyperbolic Saint-Venant system. Namely, there exist not a lot of harmonics with different velocities, and then system reaches its equilibrium (in the case of Saint-Venant the solution is just two opposite velocities, without dispersive effects). However we can see the difference of phase and group velocities in the both cases. 

In order to check numerically the order of approximation of the numerical schemes, we have constructed the reference solution for velocity. The fundamental solution of \eqref{KN-oneeq} can be written as
\begin{equation*}
		w_{ref}(t,x) = \FF^{-1} \left(\cos(\frac{\xi t}{\sqrt{1+\xi^2\varepsilon}} \ast \FF\left(w_0(x)\right)\right),
\end{equation*} 
here $\FF$, $\FF^{-1}$ are Fourier and inverse Fourier transform, and $w_0(x)$ initial data. For numerical test, the reference solution is calculated by using Fast Fourier transform and periodic boundary conditions. The extent of the computational domain is chosen large enough to avoid any spurious effects of the boundary conditions. The evolution of reference solution is shown on figure \ref{ref}. 

\begin{figure}[t]
	\begin{center}
		\begin{minipage}[h]{0.45\linewidth}
			\center{\includegraphics[width=1\linewidth]{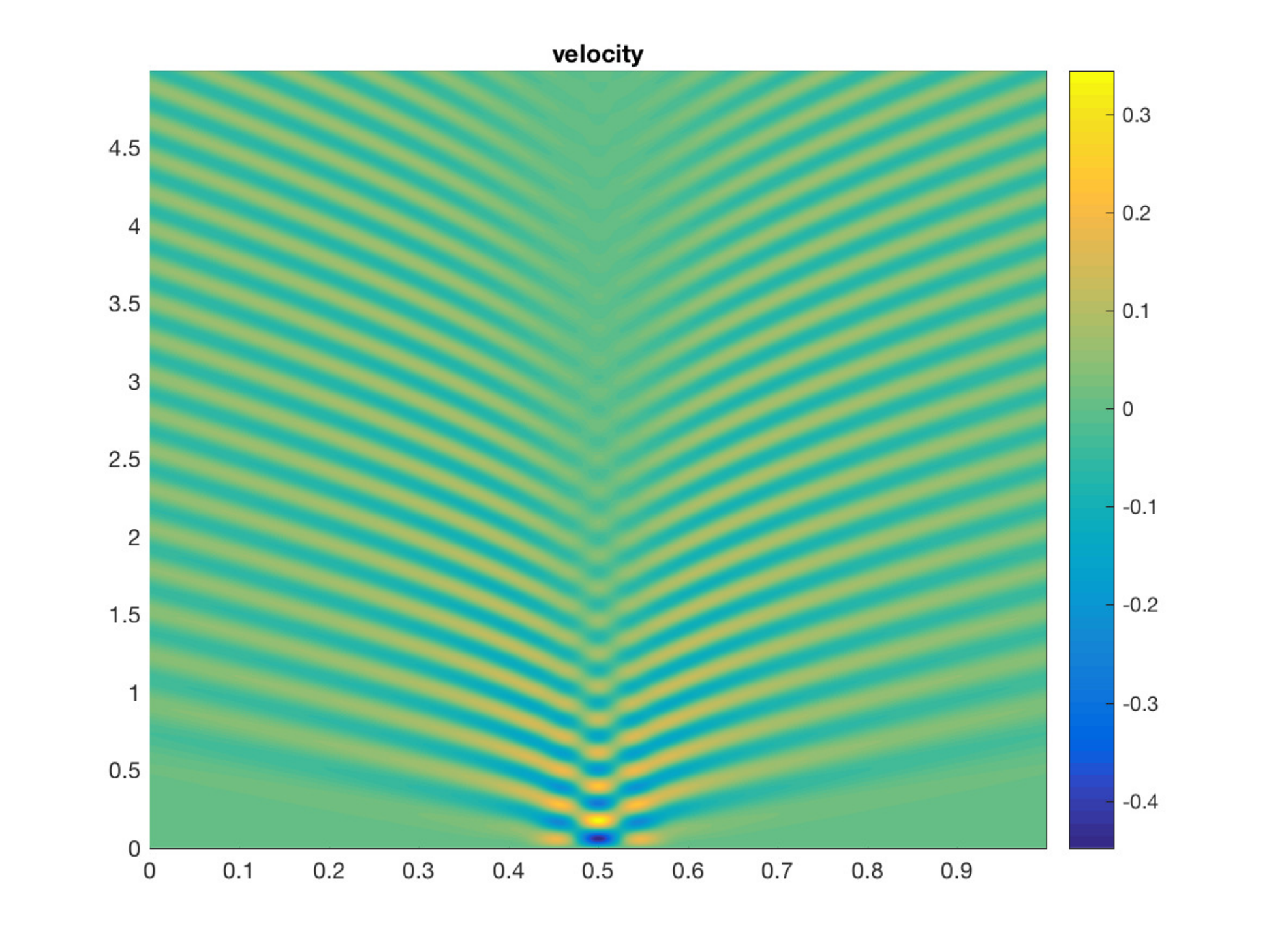}\\ {$\varepsilon = 10^{-3}$}}
		\end{minipage}
		\hfill
		\begin{minipage}[h]{0.45\linewidth}
			\center{ \includegraphics[width=1\linewidth]{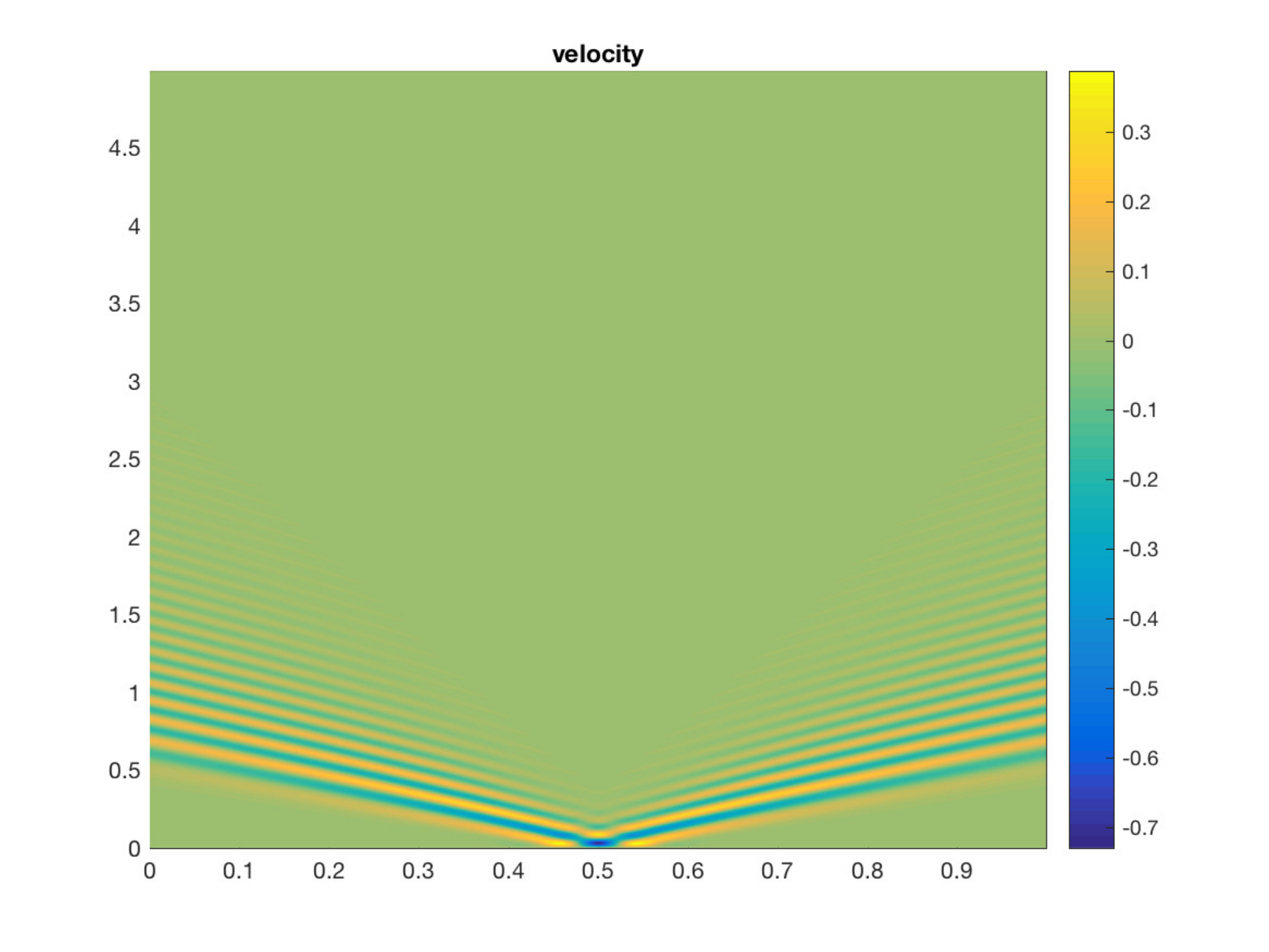}\\ {$\varepsilon = 10^{-4}$}}
		\end{minipage}
	\end{center}
	\caption{Numerical solution on a Staggered(up) and Collocates(down) grids: Evolution velocity profile for $\delta x = 10^{-3}$, $\delta t = 10^{-2}$ with \eqref{iniWP} initial datum.}
	\label{WP}
\end{figure}
We define the error functions of approximation which corresponds to the discrete version of $L^t_\infty L^x_2$ and  $L^t_2 L^x_2$ norms of the errors. Let us first denote
\begin{equation*}
	e_n = \| w(t_n,\cdot) -  w_{ref}(t_n,\cdot) \|_{L_2},
\end{equation*}
for all time step $t_n$, then the discrete norms are defined as follows
\begin{equation*}
	L_2err = \left(\dt \sum_{n = 1}^{N} (e_n^2)\right), \quad L_\infty err = \max\limits_{0< n < N} \big(e_n\big).
\end{equation*}

The next estimations are satisfied due to second order for both numerical scheme on a staggered and collocated grid
\begin{equation*}
L_2err = C^2_t \dt^2 + C^2_x \dx^2, \quad L_\infty err = C^\infty_t \dt^2 + C^\infty_x \dx^2,
\end{equation*}
where $C^2_{t,x}$, $C^\infty_{t,x}$ are universal constant. We start the analysis of the behavior of error functions with respect to $\dx$. For that purpose, we take $N = 10^{3}$ which leads to value for $\dt$ small enough to be sure that the dominating error term is linked to $C_x$ . The errors are plotted on figure \ref{err}. The second order accuracy with respect to space step is satisfied.

In order to check the approximation order with respect to $\dt$, we fix $J = 2^{15}$, to take $\dx$ small enough and be sure that there is no influence of $C^2_x$, $C^\infty_x$. We find the second order of approximation as well. The plots are presented on figure \ref{errdt}.

\begin{figure}[t]
	\begin{center}
		\begin{minipage}[h]{0.49\linewidth}
			\center{ \includegraphics[width=1\linewidth]{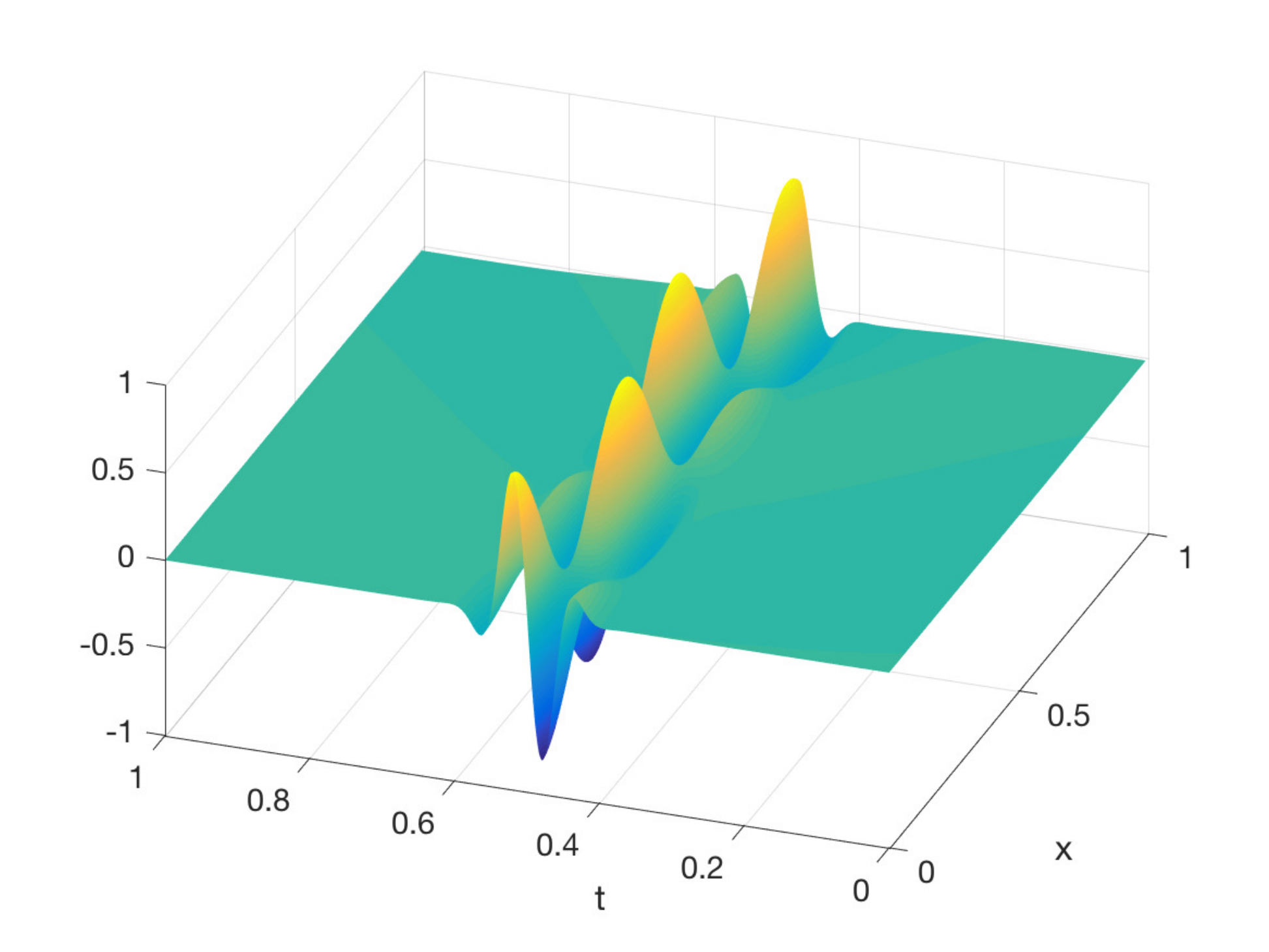}  \\ {$\varepsilon = 10^{-2}$}}
		\end{minipage}
		\hfill
		\begin{minipage}[h]{0.49\linewidth}
			\center{\includegraphics[width=1\linewidth]{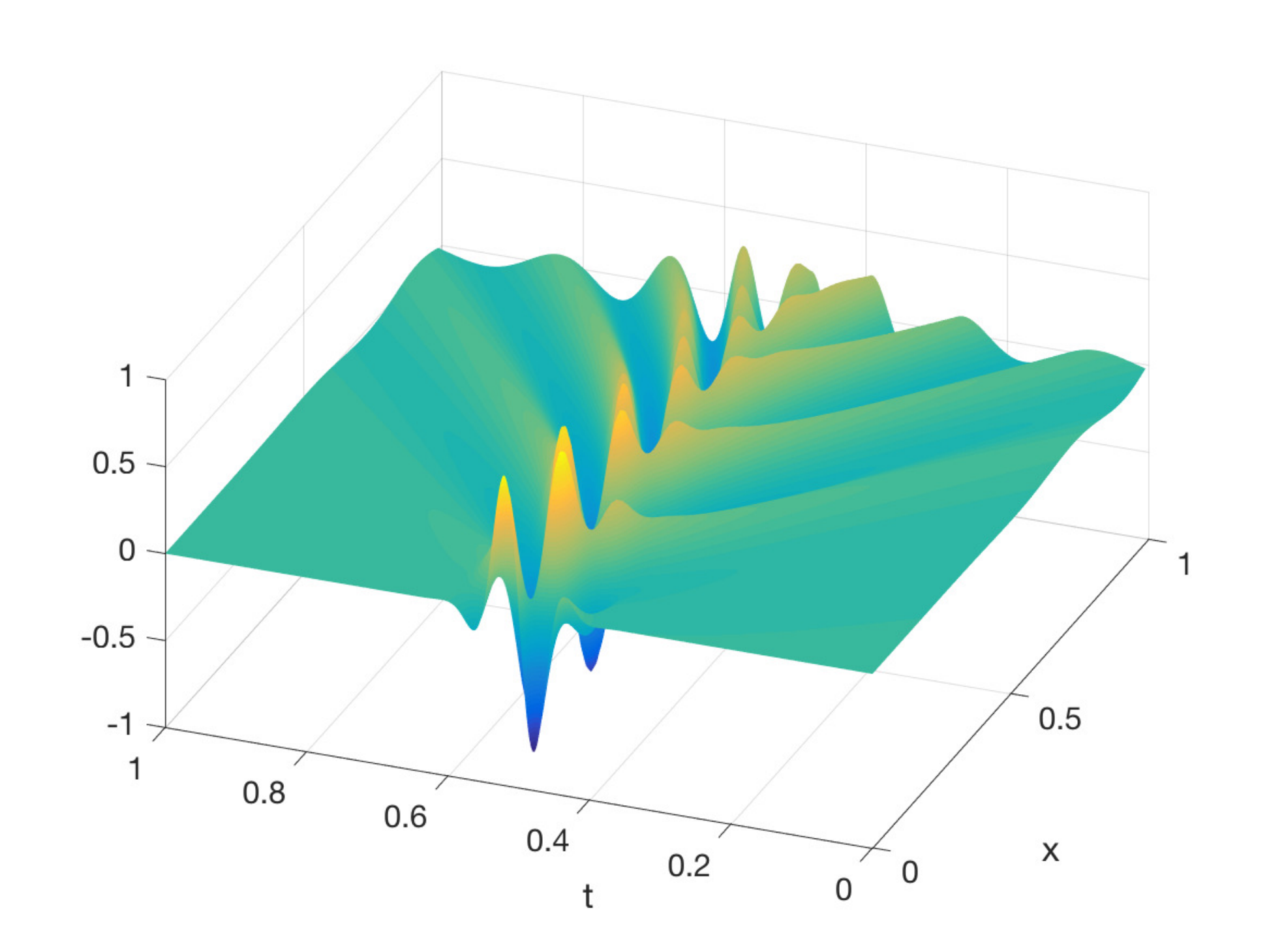}\\ {$\varepsilon = 10^{-3}$}}
		\end{minipage}
	\end{center}
	\caption{Evolution of the reference solution for $\varepsilon = 10^{-2}$ (left) and $\varepsilon = 10^{-3}$ (right).}
	\label{ref}
\end{figure}

\begin{figure}[t]
		\begin{minipage}[h]{0.48\linewidth}
			\center{ \includegraphics[width=1\linewidth]{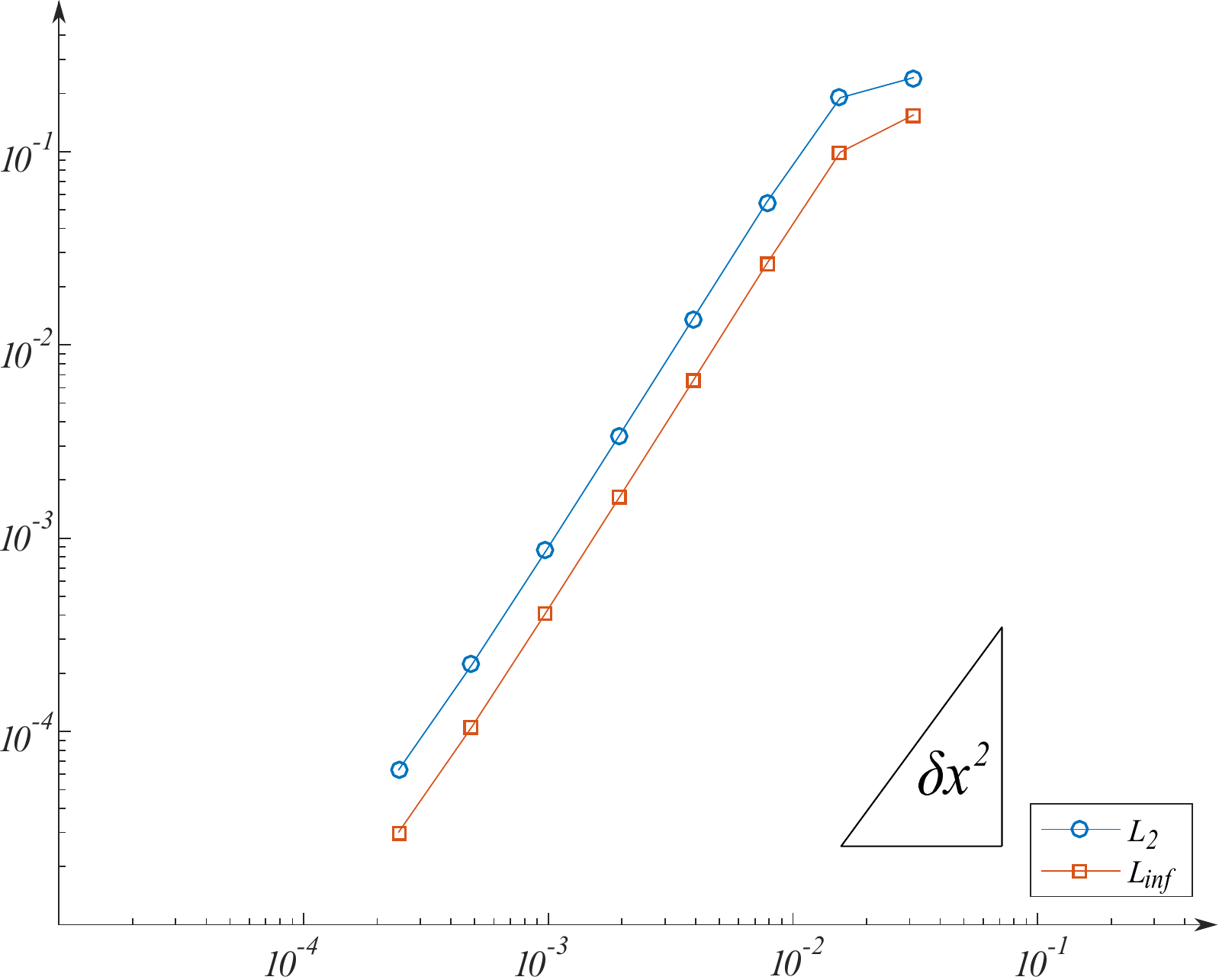}  \\ {$\varepsilon = 10^{-2}$}}
		\end{minipage}
		\hfill
		\begin{minipage}[h]{0.48\linewidth}
			\center{\includegraphics[width=1\linewidth]{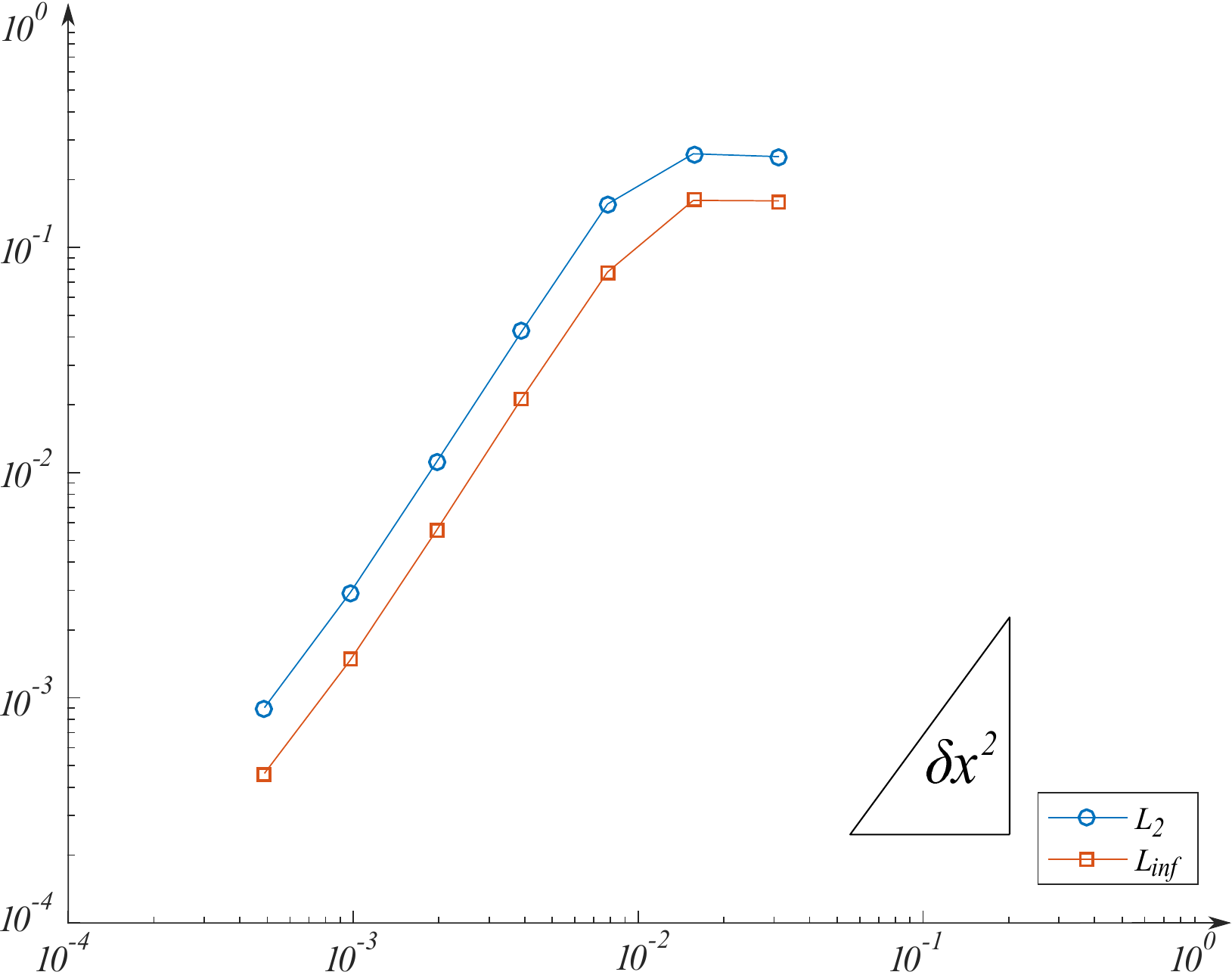}\\ {$\varepsilon = 10^{-3}$}}
		\end{minipage}
	\vfill
	\begin{minipage}[h]{0.48\linewidth}
		\center{ \includegraphics[width=1\linewidth]{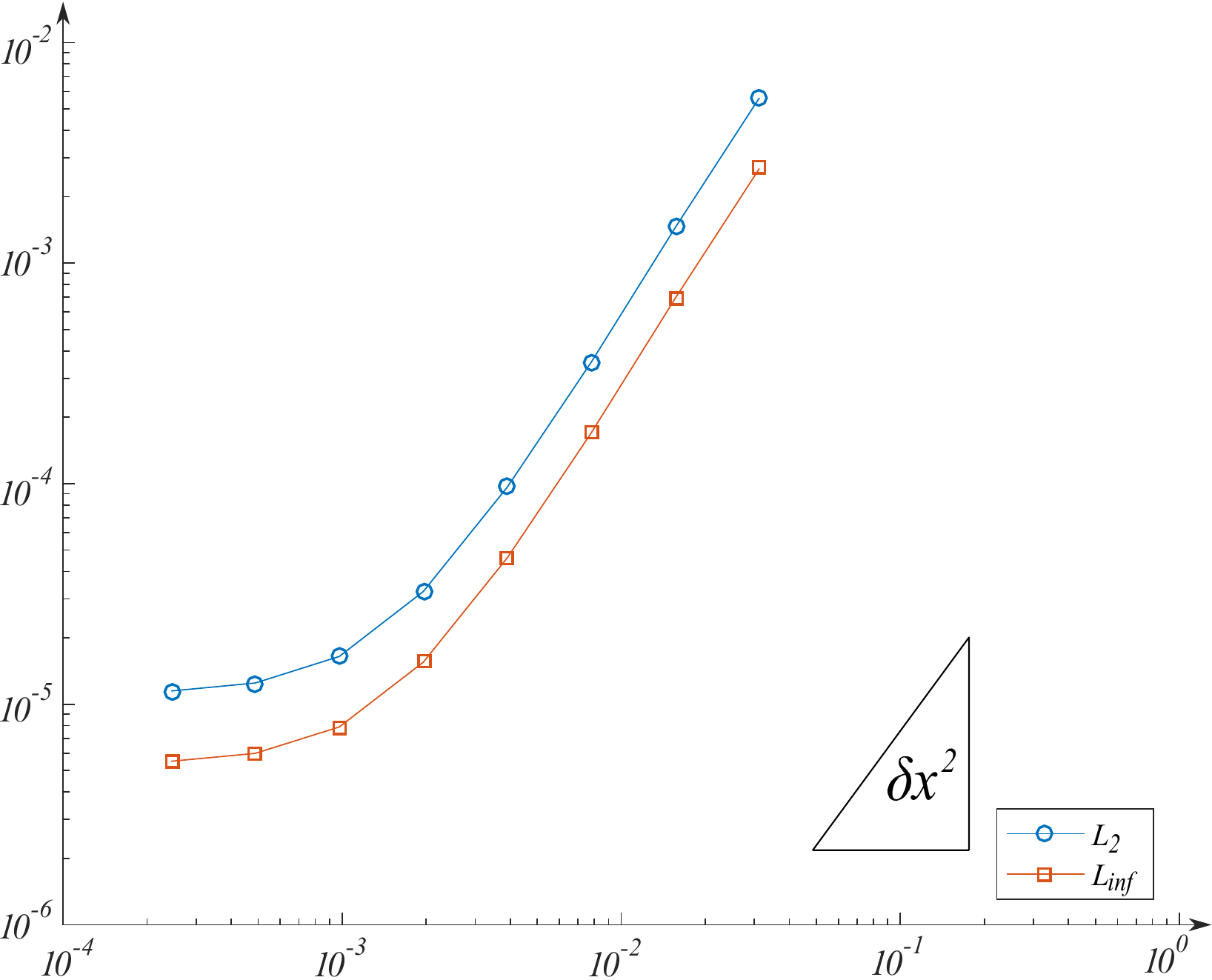}  \\ {$\varepsilon = 10^{-2}$}}
	\end{minipage}
	\hfill
	\begin{minipage}[h]{0.48\linewidth}
		\center{\includegraphics[width=1\linewidth]{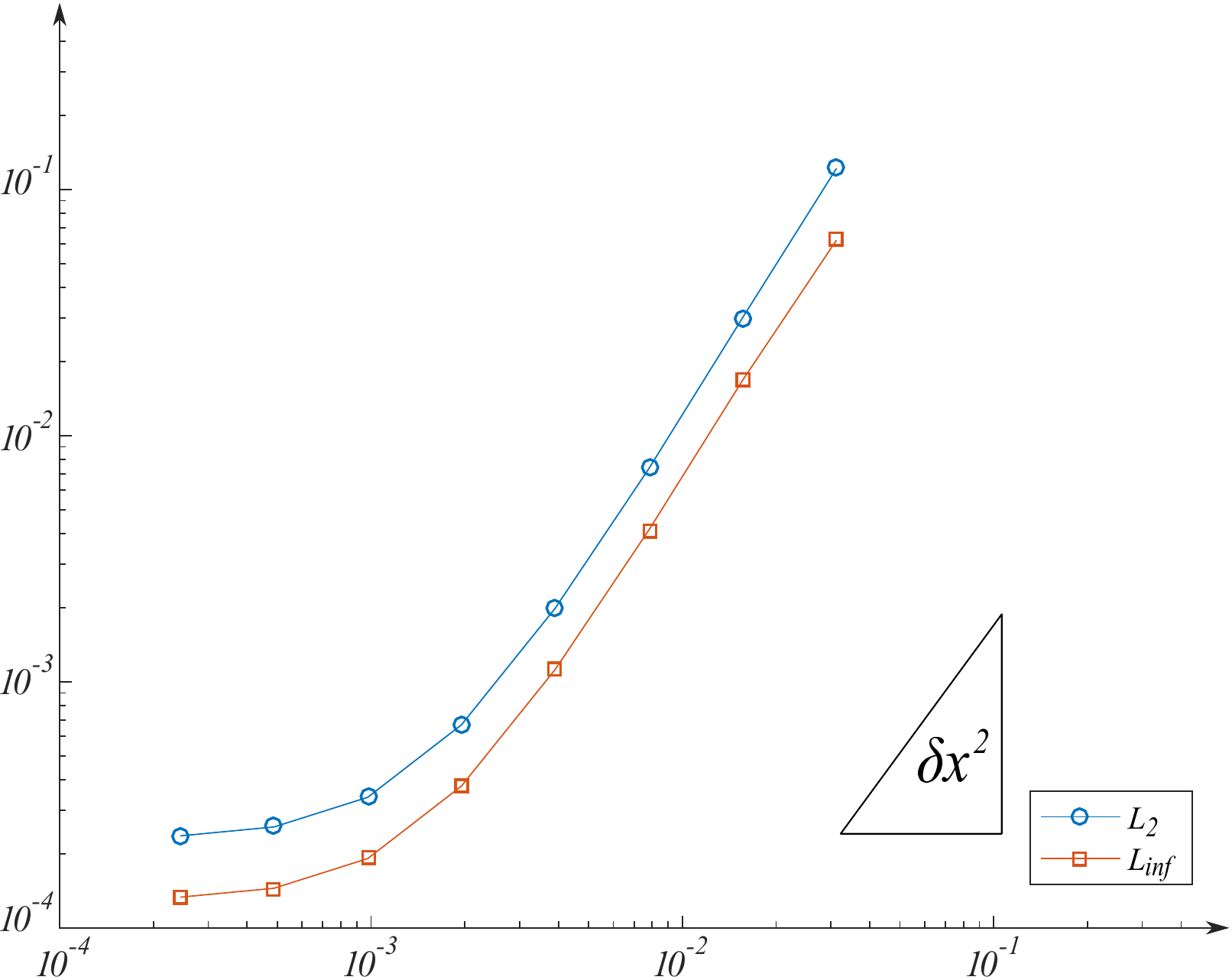}\\ {$\varepsilon = 10^{-3}$}}
	\end{minipage}
	\caption{Evolution of the error functions for numerical methods on a collocated (up) and staggered (down) grid with respect to $\dx$.}
	\label{err}
\end{figure}

\begin{figure}[t]
	\begin{minipage}[h]{0.49\linewidth}
		\center{ \includegraphics[width=1\linewidth]{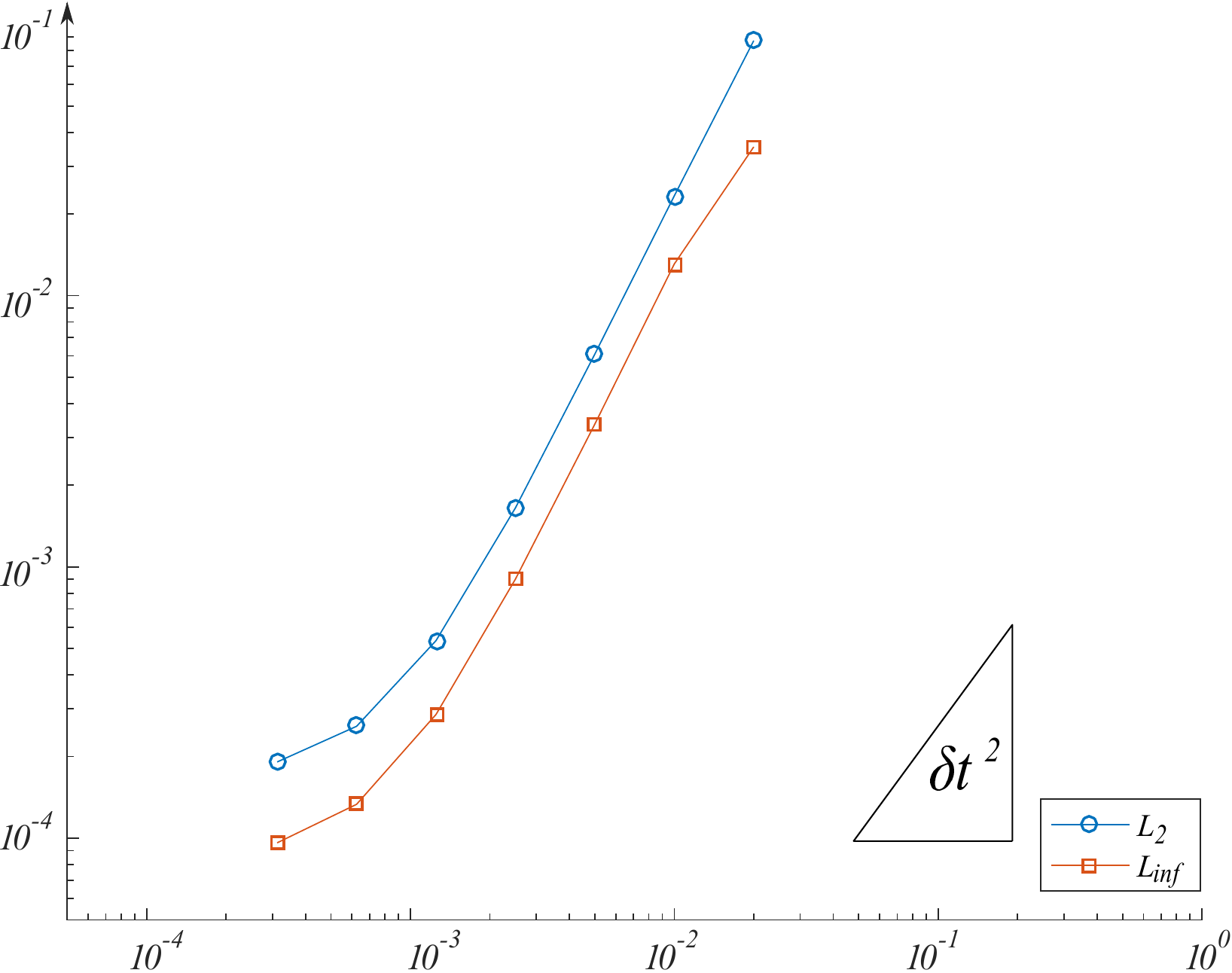}  \\ {$\varepsilon = 10^{-3}$}}
	\end{minipage}
	\hfill
	\begin{minipage}[h]{0.49\linewidth}
		\center{\includegraphics[width=1\linewidth]{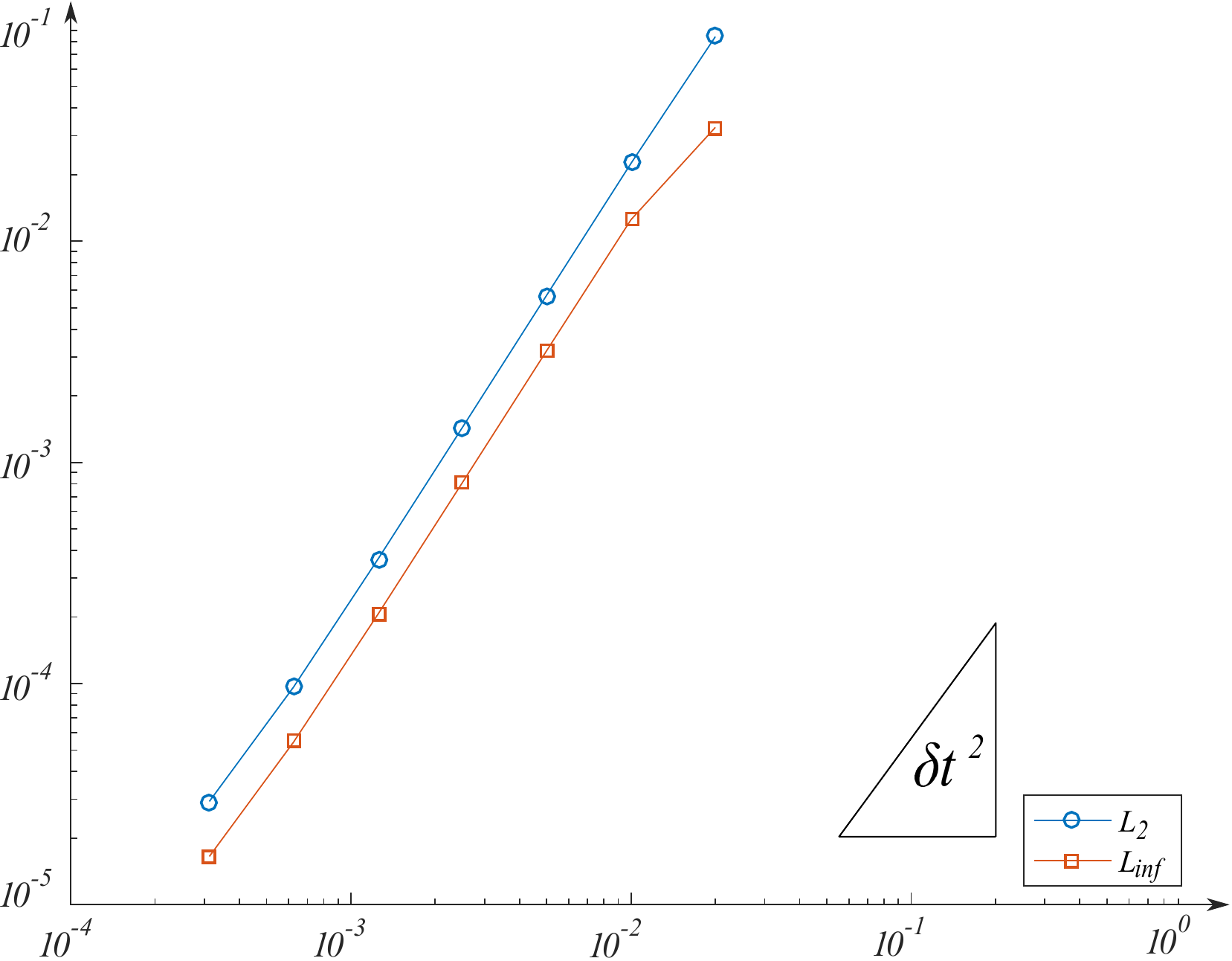}\\ {$\varepsilon = 10^{-3}$}}
	\end{minipage}
	\caption{Evolution of the error functions for numerical methods on a collocated (left) and staggered (right) grid with respect to $\dt$.}
	\label{errdt}
\end{figure}

 \subsection{Incoming wave}
 
 In this subsection we will consider the numerical test with travelling wave coming into the computational domain, which is an important real physical case. We follow here the method presented in \cite{AMP} for the Schr\"{o}dinger-Poisson system and successfully applied in \cite{BMN} for Benjamin-Bona-Mahoney equation. 
 
 Let us denote $w^{in}(x,t)=\beta\cos(kx-\omega(k)t)$ a plane wave solution for the velocity of the linear equation \eqref{KN-oneeq}. Now we are searching for transparent boundary conditions for the linear equation with an initial data $w_0$ satisfying $w_0(x) = w^{in}(x)$, $\forall x \leq x_l$ and $w_0(x) = 0$, $\forall x \geq x_r$. For that purpose, we decompose $w$ as $w(x, t) = \chi(x)w^{in}(x, t) + v(x, t)$, where the cut-off function $\chi$ is defined as $\chi = 1, \forall x \leq x_l$ $\chi = 0, \forall x \geq x_r$,  new unknown function $v$ is compactly supported in $[x_\ell, x_r]$. For $v$ one finds the following equation with a source term:
 \begin{equation*}
 \begin{array}{c}
	(v - v _{xx})_{tt} - v_{xx} = G_\varepsilon(x,t)\\
	G_\varepsilon(x,t) = \varepsilon (\chi''(x) w^{in}_{tt}(t,x) + 2 \chi' w^{in}_{xtt}(t,x) ) + \chi''(x) w^{in}(t,x)  + 2 \chi''(x) w^{in}_{x}(t,x) 
 \end{array}
 \end{equation*}
 The derivation of continuous boundary condition for $v$ is exactly similar to the homogeneous case ($w^{in} = 0$) discussed above, one finds
  \begin{equation}\label{KN-condin}
 \begin{array}{c}
 \dsp  w_x(t, x_r) = -\partial/\partial t \int_{0}^{t} \mathcal{J}_0( s/\sqrt{\varepsilon} ) w(t-s,x_r)ds ,\\[4mm]
 \dsp  \partial_x(w- w^{in})(t, x_l)  =   \partial/\partial t \int_{0}^{t} \mathcal{J}_0( s/\sqrt{\varepsilon} ) (w- w^{in})(t-s,x_l)ds.
 \end{array}
 \end{equation}

 For the discrete boundary condition the construction procedure repeats the method proposed above as well. The continuous plane wave solution is replaced by the discrete solution
 \begin{equation*}
		\dsp w^{in}_{n,j} = \beta \cos(j k \dx -  n \tilde{\omega}(k) \dt), \quad \tilde{\omega}(k) = \frac{1}{\dt} \arccos\left( \frac{2 \dx^2 + (4\varepsilon - \dt^2) \sin^2(k\dx/2)}{2 \dx^2 + (4\varepsilon + \dt^2) \sin^2(k\dx/2)}\right),
 \end{equation*}
 and condition on the left is written as 
 \begin{multline}\label{KN-leftin}
 \Lambda (w_1^{n+1} - [w^{in}]_1^{n+1} ) - (\Lambda + \dx^2 + 2 \dx \sqrt{\Gamma})(w_0^{n+1} - [w^{in}]_0^{n+1})  = \\[3mm]  2(\mu (w_1^{n} - [w^{in}]_1^{n}) - (\mu + 2 \dx^2 + \dx \sqrt{\Gamma} (v + 1))(w_0^{n} - [w^{in}]_0^{n})) - \\[3mm]  (\Lambda +(w_1^{n-1} - [w^{in}]_1^{n-1} ) - (\Lambda + \dx^2 + 2 \dx \sqrt{\Gamma})(w_0^{n-1} - [w^{in}]_0^{n-1})) +\\[3mm]
 2 \dx\sqrt{\Gamma} \left( (\PP_2 - 2 v^2 + v)  (w_0^{n-1} - [w^{in}]_0^{n-1}) + \sum_{k =2}^{n} s_k  (w_0^{n-k} - [w^{in}]_0^{n-k}) \right),
 \end{multline}
and on the right,
 \begin{multline}\label{KN-rightin}
 \Lambda (w_{J+1}^{n+1} - [w^{in}]_{J+1}^{n+1} ) - (\Lambda + \dx^2 - 2 \dx \sqrt{\Gamma}) (w_{J}^{n+1} - [w^{in}]_{J}^{n+1}) = \\[3mm]  2(\mu (w_{J+1}^{n} - [w^{in}]_{J+1}^{n}) - (\mu + 2 \dx^2 - \dx \sqrt{\Gamma} (v + 1))(w_J^{n} - [w^{in}]_{J}^{n}) - \\[3mm] 
  - (\Lambda (w_{J+1}^{n-1} - [w^{in}]_{J+1}^{n-1}) - (\Lambda + \dx^2 - 2 \dx \sqrt{\Gamma})(w_J^{n-1} - [w^{in}]_{J}^{n-1}) - \\[3mm] 
- 2 \dx\sqrt{\Gamma} \left( (\PP_2 - 2 v^2 + v)  (w_J^{n-1} - [w^{in}]_{J}^{n-1}) + \sum_{k =2}^{n} s_k(v)  (w_J^{n-k} - [w^{in}]_{J}^{n-k}) \right).
 \end{multline}
 Conditions for the system \eqref{KN-CrNicDisCol} can be written in the same manner. 

The numericals results are presented on the Figure,  \ref{inc}. We put wave number $k = 2 \pi p$, $p \in N$. And we presented the results for different wave number ($p = 4, 8$). In both case there exist a transient regime, but after the wave solution propagates correctly. We observe again the difference between phase and group velocities. Note that the characteristics in the $(x,t)$ plane have all a slope close to 1 in the zone after transition, which corresponds to the velocity of the waves (a coefficient preceding  $w_x$). But the part of energy is carried along the characteristic with the smaller slope on the border of the transient regime. Which corresponds to the fact that group velocity is smaller.

\begin{figure}[t]
	\begin{minipage}[h]{0.49\linewidth}
		\center{ \includegraphics[width=1\linewidth]{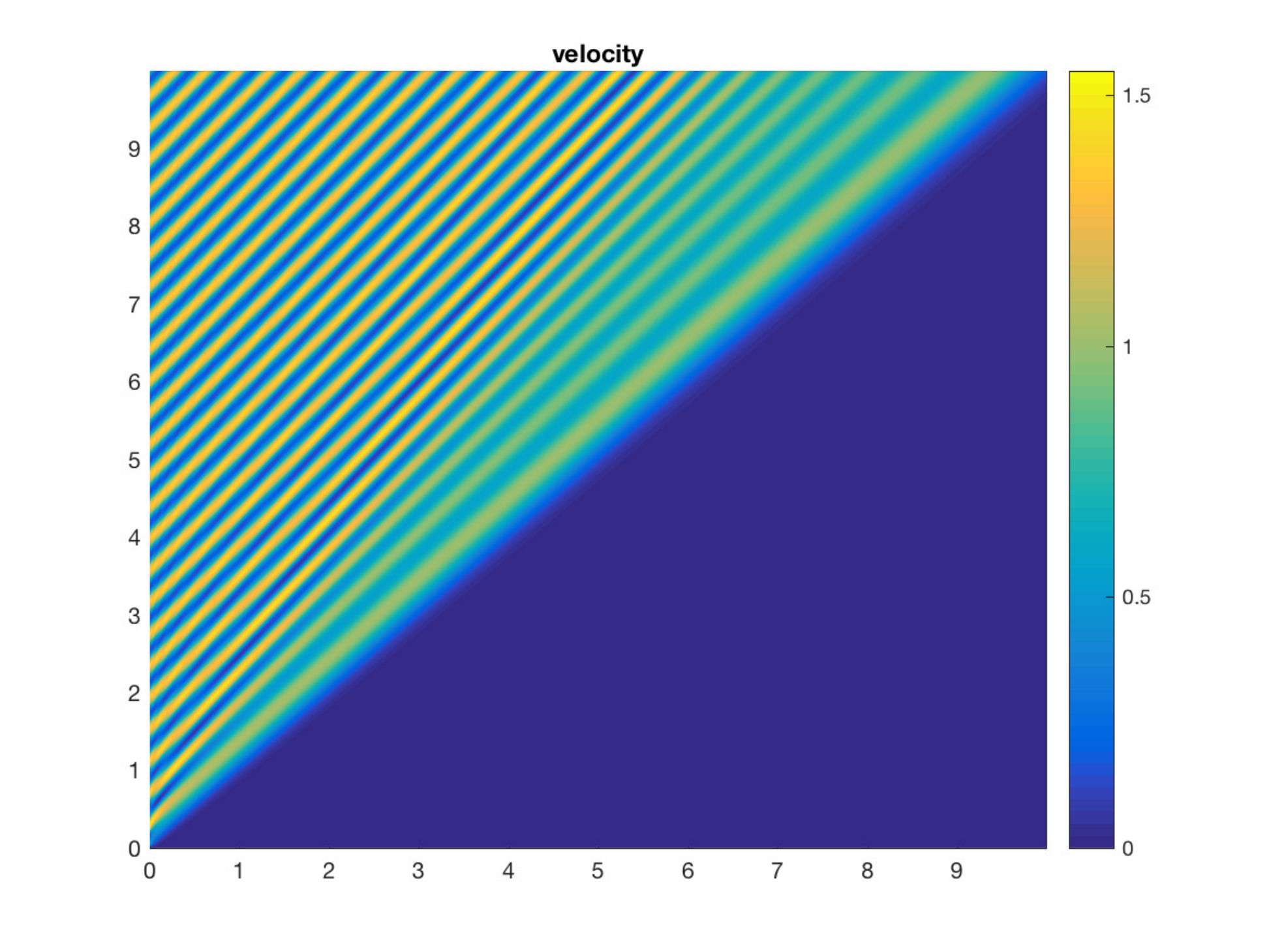}  \\ {$p = 4, \varepsilon = 10^{-3}$}}
	\end{minipage}
	\hfill
	\begin{minipage}[h]{0.49\linewidth}
		\center{\includegraphics[width=1\linewidth]{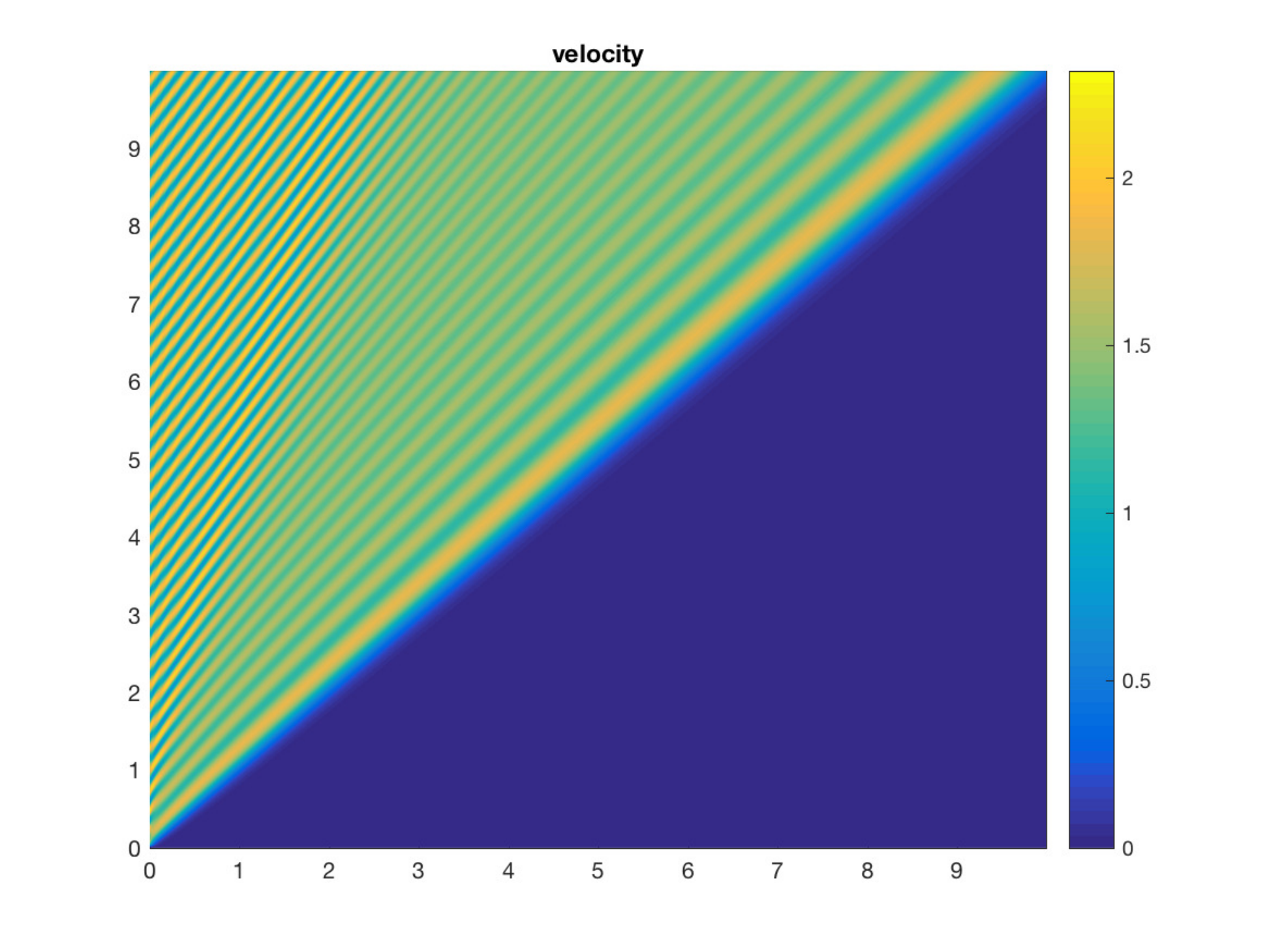}\\ {$p = 8, \varepsilon = 10^{-3}$}}
	\end{minipage}
	\caption{Evolution of incoming wave solution for different wave number.}
	\label{inc}
\end{figure}

 \section{Conclusion}

In this paper, we derived exact and discrete transparent boundary conditions for the linear Green-Naghdi
system for a Crank Nicolson discretization on a staggered and collocated grid.  Both schemes are proved to be stable, consistent and convergent. The technique is validated numerically as well for outgoing wave with the different initial data. We show how to deal with the problem of wave generation in water wave problems and prove accuracy of the proposed method on the numeric test. 

In practice, we will have to deal with non-linear equations. It remains an open question what are the transparent boundary conditions for this case? One can imagines to adapt our strategy to linear equations with variable coefficients and then adopt a fixed point strategy, as it was done for nonlinear Schrodinger equations in \cite{AABES}. An other question of interest is to derive discrete transparent boundary conditions in the case of the two-layer Green-Naghdi equations which are used to describe an internal wave propagation.


\begin{thebibliography}{ABS}
	
	\bibitem{AMP} 
	N.\,B.\,Abdallah, F.\,M\'{e}hats, O.\,Pinaud {\em On an open transient Schr\"{o}dinger-Poisson system}, Math. Models Methods Appl. Sci. 15 (2005), 667.
	
	\bibitem{AABES} 
	X.\,Antoine, A.\,Arnold, C.\,Besse, M.\,Ehrhardt, and A. Schadle {\em A review of transparent and artificial boundary conditions techniques for linear and nonlinear Schrodinger equations}, Commun. Comput. Phys., 4 (2008), 729-796.
	
	\bibitem{Arn}
A. Arnold,
{\it Numerically absorbing boundary conditions for quantum evolution equations}, VLSI Design, 6 (1998), 313-319.

\bibitem{AES}
A. Arnold, M. Ehrhardt and I. Sofronov, 
{\it Discrete transparent boundary conditions for the Schr{\"o}dinger equation: Fast calculation, approximation, and stability}, Communications in Mathematical Sciences, 3 (2003), 501-556.
	
	\bibitem{BEL-V}
	C.\,Besse, M.\,Ehrhardt, I.\,Lacroix-Violet {\em Discrete artificial boundary conditions for the linearized Korteweg–de Vries equation}, Num.Meth. for PDE, V. 32, Issue 5, (2016) 1455-1484.
	
	\bibitem{BMN}
	C.\,Besse,  B.\,Mesognon, P.\,Noble  {\em Discrete Artificial Boundary Condition for the Benjamin- Bona-Mahoney equation}, Preprint 2016, hal-01305360.
	
	\bibitem{BNS}
	C.\,Besse, P.\,Noble, D.\,Sanchez {\em Discrete transparent boundary conditions for the mixed KDV-BBM equation.}, Preprint arXiv:1609.08941
	
	\bibitem{ET}
	M. Ehrhardt {\em Discrete Artificial Boundary Conditions}, 2001.
	
	\bibitem{GN} 	
	E. Green, P. M. Naghdi {\em A derivation of equations for wave propagation in water of variable depth}, J. Fluid Mech. 78 (1976): 237.
	
	\bibitem{L} 
	D. Lannes {\em The Water Waves Problem: Mathematical Analysis and Asymptotics}, Amer. Mathematical Society (2013): 188.
	
	\bibitem{SV} 
	A.J.C. de Saint Venant  {\em Th\'{e}orie du mouvement non-permanent des eaux, avec application aux crues des rivières et à l’introduction des marées dans leur lit.} C.R. Acad. Sc. Paris, 73 (1871):147–154.
	

	
	\bibitem{ZWH}
	C.\,Zheng, X.\,Wen, and H.\,Han {\em Numerical Solution to a Linearized KdV Equation on Unbounded Domain}, Numer. Meth. Part. Diff. Eqs. 24 (2008), 383-399.
	

	
\end{thebibliography}
\end{document}